\documentclass[12pt]{amsart}
\usepackage{amscd,amssymb}
\usepackage[arrow,matrix]{xy}
\usepackage[plainpages,backref,urlcolor=blue]{hyperref}
\theoremstyle{plain}
\newtheorem{thm}[subsection]{Theorem}
\newtheorem{lem}[subsection]{Lemma}
\newtheorem{prop}[subsection]{Proposition}
\newtheorem{cor}[subsection]{Corollary}
\newtheorem{sch}[subsection]{Scholium}

\theoremstyle{definition}
\newtheorem{rk}[subsection]{Remark}
\newtheorem{definition}[subsection]{Definition}
\newtheorem{ex}[subsection]{Example}

\numberwithin{equation}{section}
\newcommand{\CB}{{\mathcal B}}

\newcommand{\G}{{\Gamma}}
\newcommand{\g}{{\gamma}}

\newcommand{\A}{{\mathcal A}}

\newcommand{\CM}{{\mathcal M}}

\DeclareMathOperator{\Id}{Id}

\DeclareMathOperator{\Fix}{Fix}
\DeclareMathOperator{\Ind}{Ind}
\DeclareMathOperator{\Ker}{Ker}
\DeclareMathOperator{\Res}{Res}
\DeclareMathOperator{\Reg}{Reg}
\DeclareMathOperator{\Sym}{Sym}

\DeclareMathOperator{\Trace}{Trace}
\DeclareMathOperator{\Tr}{Trace}

\newcommand{\CA}{{\mathcal A}}

\newcommand{\CP}{{\mathcal P}}

\newcommand\be{\begin{equation}}
\newcommand\beq{\begin{equation}}
\newcommand\ee{\end{equation}}

\newcommand\lr{{\longrightarrow\;}}

\newcommand{\wh}{\widehat}

\newcommand{\al}{{\alpha}}
\newcommand{\bt}{{\beta}}
\newcommand\wt{\widetilde}
\newcommand{\inv}{^{-1}}
\newcommand{\0}{^{\circ}}
\newcommand{\GL}{{\text{GL}}}
\newcommand{\ve}{{\varepsilon}}

\newcommand{\Z}{\mathbb{Z}}
\newcommand{\Q}{\mathbb{Q}}
\newcommand{\R}{\mathbb{R}}
\newcommand{\C}{\mathbb{C}}

\newcommand{\PP}{\mathbb{P}}

\DeclareMathOperator{\lcm}{lcm}

\DeclareMathOperator{\codim}{codim}

\newcommand{\ot}{{\otimes}}

\makeatletter
\@namedef{subjclassname@2010}{%
  \textup{2010} Mathematics Subject Classification}
\makeatother

\begin{document}

\title [ ]
{On the cohomology of the Milnor fibre of a hyperplane arrangement }

\author[Alexandru Dimca]{Alexandru Dimca$^{1,2}$}

\address{Univ. Nice Sophia Antipolis, CNRS,  LJAD, UMR 7351, 06100 Nice, France.}
\email{dimca@unice.fr}

\author[Gus Lehrer]{Gus Lehrer$^{2}$}
\address{ School of Mathematics and Statistics F07,
University of Sydney, NSW 2006, 
Australia  }
\email{gustav.lehrer@sydney.edu.au}

\thanks{$^1$ Partially supported by  Institut Universitaire de France}
\thanks{$^2$ Partially supported by Australian Research Council Grants DP0559325 and DP110103451} 

\subjclass[2010]{Primary 32S22, 32S35; Secondary 32S25, 32S55.}

\keywords{hyperplane arrangement, Milnor fibre, monodromy, equivariant Hodge-Deligne polynomial. }

\begin{abstract} We investigate the cohomology of the Milnor fibre of a reflection arrangement
as a module for the group $\Gamma$ generated by the reflections, together with the cyclic monodromy. Although 
we succeed completely only for unitary reflection groups of rank two, we establish some general results
which relate the isotypic componenents of the monodromy on the cohomology, to the Hodge structure
and to the cohomology degree. Using eigenspace theory for reflection groups, we
prove some sum formulae for additive functions such as the equivariant weight polynomial and
certain polynomials related to the Euler characteristic, such as the Hodge-Deligne polynomials.
We also use monodromy eigenspaces to determine the spectrum in some cases,
which in turn throws light on the Hodge structure of the cohomology. These methods enable us to compute
the complete story, including the representation of $\Gamma$ on the Hodge components in each
cohomology degree, for some groups of low rank.  
\end{abstract}

\maketitle

\tableofcontents

\section{Introduction and notation} \label{sec:intro} Let $G$ be a reflection group in $V:=\C^\ell$, in the sense
of \cite{LT}. Associated to $G$ we have the following data: $\CA=\CA_G$ is the set of reflecting hyperplanes of $G$;
$M=M_G=\C^\ell\setminus\cup_{H\in\CA}H$ is the corresponding hyperplane arrangement complement, and $d_1,d_2,\dots,d_\ell$ are
the invariant degrees of $G$ (see \cite[Ch. 3]{LT}). For each hyperplane $H\in\CA$, let $e_H$ be the order of the 
pointwise stabiliser of $H$, and let $\ell_H\in V^*$ be a linear form such that $H=\Ker(\ell_H)$.
For simplicity we shall take $G$ to be irreducible. Refer to \cite{OT} 
for general notions and results concerning the hyperplane arrangements.

It is well known that the polynomial $Q(z):=\prod_{H\in\CA}\ell_H^{e_H}$ ($z\in V$) is invariant under $G$.
\begin{definition}\begin{enumerate}
\item 
The {\it Milnor fibre} $F=F(\CA)$ of the arrangement $\CA$ is the variety $F:=Q\inv(1)\subset V$.
\item The {\it reduced Milnor fibre} $F_0=F_0(\CA)$ of the arrangement $\CA$ is the variety $F_0:=Q_0\inv(1)\subset V$,
where $Q_0(z):=\prod_{H\in\CA}\ell_H$.
\end{enumerate}
\end{definition}
 
Let $m$ be the degree of $Q$, and let $\Gamma:=G\times\mu_m$, where $\mu_m$ is the group of $m^{\text th}$
roots of unity in $\C$. Then $\Gamma$ acts on $V$ via $(g,\xi)(v)=\xi\inv gv$ ($g\in G, \xi\in\mu_m, v\in V$),
and it is evident that $\Gamma(F)\subseteq F$, so that $\Gamma$ acts on $F$. This (left) $\G$-action can be regarded as a 
group homomorphism $\tau: \G \to Homeo(F)$ of $\G$ into the group of homeomorphisms of $F$, defined by $\tau(\gamma)(x)=\gamma x.$ 
This in turn induces a group homomorphism
\be \label{eq:cohoaction}
\phi: \G \to Aut(H^*(F,\C)), \text{ given by } \phi(\gamma)=(\tau(\gamma)^*)^{-1}.
\ee
Ultimately our goal is to elucidate the cohomology $H^*(F,\C)$ as a module for this $\Gamma$-action. 

Note that $F_0$ has a smaller symmetry group $\Gamma_0\subseteq\Gamma$, but we show, using some
simple arguments concerning induced representations, that to answer the
above question, it suffices to determine the action of $\Gamma_0$ on $H^*(F_0,\C)$.  This brings into
play some general considerations about reduced hypersurface singularities.

In the next section we use eigenspace theory for unitary reflection groups to give a sum
formula for any `additive function' on $F$. An additive function (with values in any abelian group,
which we shall usually take to be polynomials with coefficients in a Grothendieck ring) on a topological space is
a function $\alpha$ such that if $X=C\amalg U$, where $U,C$ are respectively open and closed,
then $\alpha(X)=\alpha(C)+\alpha(U)$. Particular cases include the equivariant weight polynomial,
and the equivariant Euler characteristic, which is the virtual  module 
\be\label{eq:eulerchar}
\chi^\Gamma(F):=\sum_{j}(-1)^jH^j(F,\C)\in R\Gamma, 
\ee
in the Grothendieck ring of $\Gamma$; this latter case has been treated in \cite{DeL}.

Section 3 presents a complete solution for the groups of rank two. Of course the Euler characteristic suffices for the
$\Gamma$-action on $H^1$, but we give a more detailed analysis of the Hodge components (partly in Section 5), and use some results of Orlik and
Solomon on isolated singularities to give more explicit formulae. In section 4, we give some general results about connections between the monodromy
action and cohomology degree, and use these to give some general results about the symmetric group case. The
section concludes with a complete solution of the case where $G=\Sym_5$.

In section 5, more detail is entered into concerning the mixed Hodge structure on the cohomology spaces.
We introduce the notions of equivariant Hodge-Deligne polynomials and Poincar\'e-Deligne polynomials, 
and show how the spectrum of a hyperplane arrangement relates to our question. Using recent results 
of various authors on the computation of the spectrum, we are able to give the $\mu_m$-equivariant Hodge-Deligne
polynomial of the (essential) hyperplane arrangement of type $A_n$ for $n=2,3$ or $4$, where $m=\frac{n(n+1)}{2}$.
We also give the $\Gamma$-equivariant Poincar\'e-Deligne polynomial for $A_n$ with $n\leq 3$.

Further, we give explicitly the $\Gamma$-equivariant weight polynomials for the dihedral groups, and the 
monodromy-equivariant Hodge-Deligne polynomials for all groups of rank two.

Finally, also in \S 5, we prove some general results, complementary to those concerning the relationship between monodromy
order and cohomology degree, which relate the monodromy action to the Hodge structure of the cohomology.

\section{Additive functions and the Euler characteristic}

\subsection{Free action and eigenspaces}
Observe that if $v\in F$ and $(g,\xi)\in\Gamma$, then $(g,\xi)v=v\iff gv=\xi v$. Thus $v\in F$ is fixed by some
element of $\Gamma$ if and only if $v\in V(g,\xi)$, for some $g\in G,\xi\in\C^\times$, where $V(g,\xi)$ is the 
$\xi$-eigenspace of $g\in G$. Now every element $v\in F$ lies on no reflecting hyperplane of $G$, and is therefore
regular. Conversely, every regular eigenspace intersects $F$, because $F$ spans $M$. The following statement
is immediate from this observation.

\begin{lem}\label{lem:reg}
Let $d$ be a regular number for $G$; that is, if $\zeta_d=\exp(\frac{2\pi i}{d})$, then
any $\zeta_d$-eigenspace of $g\in G$ contains a regular vector (and is therefore maximal). Then $d$ divides $m$.
\end{lem}
\begin{proof}
We give two proofs.
The first is as follows. Let $E:=V(g,\zeta_d)$ be a regular eigenspace. Then there is a vector $v\in E\cap F$.
Since $F$ is $G$-stable (since $Q$ is $G$-invariant), it follows that $gv=\zeta_d v\in F$. Thus
$Q(\zeta_dv)=\zeta_d^mQ(v)=1=Q(v)$, whence $\zeta_d^m=1$, and so $d|m$.

The second proof uses the theory of coexponents.
It follows from \cite[Prop. 4.6]{LSp} that if $d$ is regular and if the exponents and coexponents of
$G$ are respectively denoted $m_1,\dots,m_\ell$ and $m_1^*,\dots,m_\ell^*$, then modulo $d$, the two
(multi)sets $\{m_1+1,\dots,m_\ell+1\}$ and $\{-m_1^*+1,\dots,-m_\ell^*+1\}$ are equal. It follows that 
$\sum_{i=1}^\ell(m_i+m_i^*)\equiv 0(\text{mod }d)$. But it is well known (see, e.g. \cite[p.205]{LT})
that $\sum_{i=1}^\ell m_i=n_G$, the number of reflections in $G$, while $\sum_{i=1}^\ell m_i^*=N_G=|\CA_G|$.
Hence $\sum_{i=1}^\ell(m_i+m_i^*)=\sum_{H\in\CA}e_H\equiv 0(\text{mod }d)$.
\end{proof}

The basic facts concerning regular eigenspaces may be found in \cite[\S 11.4]{LT}. The main facts we require are 
as follows. If $E:=V(g,\zeta)$ contains a regular vector, then $E$ is a maximal 
$\zeta$-eigenspace, i.e. it is not properly contained in $V(x,\zeta)$ for $x\in G$; the centraliser $C_G(g)$ 
acts faithfully on $E$ as a reflection group with invariant degrees $\{d_i\mid d\text{ divides }d_i\}$, and
if $\zeta$ has order $d$, then $g$ has order $d$, and if $V(g',\zeta)$ is another regular eigenspace,
then $g'$ is conjugate to $g$ in $G$. For a regular element $g\in G$ of order $d$, we write 
$G(d):=C_G(g)$. This determines $G(d)$ up to conjugacy in $G$.
If $G$ is irreducible (as we have assumed) then $G(d)$ is irreducible for 
each regular $d$.

Now suppose that $\zeta=\zeta_d\in\C^\times$ is regular, as above. Then $d$ divides 
$\bar d:=\gcd\{d_i\mid d\text{ divides }d_i\}$. Moreover there is an element $\bar g\in G$ such that
$V(\bar g,\zeta_{\bar d})\neq 0$, where $\zeta_{\bar d}=\exp(\frac{2\pi i}{\bar d})$. Thus
$V(\bar g,\zeta_{\bar d})= V((\bar g)^{\frac{\bar d}{d}}, \zeta_d)$, since both have the same dimension.
Hence it suffices to consider regular numbers $d$ such that $d=\bar d$. Let $\CP$ denote the set
of such integers. Note that $\bar 1=\gcd\{d_1,\dots,d_\ell\}$.

If $d,e\in\CP$ and $d|e$, then for any regular (and hence maximal) $\zeta_e$-eigenspace $E_e=V(g,\zeta_e)$,
we have $E_e\subseteq V(g^{\frac{e}{d}},\zeta_d):=E_d$. It follows that if $V(d)$ is the union of all 
$\zeta_d$-eigenspaces for $d\in\CP$ and $F(d)=V(d)\cap F$, then $d|e$ implies that $F(e)\subseteq F(d)$.

The following cyclic subgroups of $\Gamma$ play an important role in the discussion below.

\begin{definition}\label{def:rd}
For $d\in \CP$, let $(g_d,\zeta_d)\in\Gamma$ be such that $V(g_d,\zeta_d)\neq 0$. This determines
$(g_d,\zeta_d)$ up to conjugacy in $\Gamma$. Define $R(d)$ to be the cyclic group $\langle(g_d,\zeta_d)\rangle$
of $\Gamma$.
\end{definition}


\begin{prop}\label{prop:decomp}
We have 
\begin{enumerate}
\item $F=\cup_{d\in\CP}F(d)$.
\item If $d|e$, $d,e\in\CP$, then $F(d)\supseteq F(e)$.
\item For any two integers $e_1,e_2$, we have $F(e_1)\cap F(e_2)=F(\lcm(e_1,e_2))$.
\item For each $d\in\CP$ let $F(d)^\circ=F(d)\setminus\cup_{d|e,e\neq d}F(e)$.
Then $F=\amalg_{d\in\CP}F(d)^\circ$
and for each $d\in\CP$, $\overline{F(d)^\circ}=F(d)=\amalg_{d|e}F(e)^\circ$.
\item $\Gamma/R(\bar 1)$ acts freely on $F(\bar 1)^\circ:=F^\circ$, where $R(\bar 1)$ denotes the subgroup
of $\Gamma$ defined above.
\end{enumerate}
\end{prop}
\begin{proof}
All statements except (3) follow from the above discussion. As for (3), it is well known that if $P_1,\dots,P_\ell$
is a set of basic invariants of the group $G$ and $P_i$ is homogeneous of degree $d_i$, then 
$V(d)=\cap_{i:d\not | d_i}V(P_i)$, where $V(f)$ denotes the zero set of the polynomial $f$.
It follows that $V(e_1)\cap V(e_2)=V(\lcm(e_1,e_2))$, and (3) is now evident.
\end{proof}

\subsection{Additive functions and Euler characteristics} 
Proposition \ref{prop:decomp} shows that the closed subspaces $F(d)$ form an Eulerian collection
in the sense of \cite[Def. (2.1)]{DL}. Let $\CB$ be the Boolean algebra of constructible subsets of $F$,
and $A$ be any abelian group.
Recall \cite[loc. cit.]{DL} that a function $\beta:\CB\to A$ is additive if for any subsets $Y\supseteq Z$ in $\CB$, we have
$\beta(Y)=\beta(Z)+\beta(Y\setminus Z)$.

The decomposition of $F$ as a
disjoint union of locally closed $\Gamma$-invariant subvarieties in Proposition
\ref{prop:decomp}(4) implies that we have the following
relation for any additive function (see \cite[Prop. (2.2)]{DL}).

\begin{lem}\label{lem:decomp-gen} Let $\beta$ be any additive function on the constructible subspaces of $F$.
Then
$$
\beta(F)=\sum_{d\in\CP}\beta(F(d)^\circ).
$$
\end{lem}

Lemma \ref{lem:decomp-gen} applies to any additive function $\beta$, and in \S\ref{s:mh} below we shall 
explore examples of additive functions arising from the mixed Hodge structure on the cohomology.
By \cite[(2.6)]{DL}, one such function is the compactly supported weight polynomial $W^\Gamma_c(F,t)$ (cf. \cite[Def. (1.5)(ii)]{DL}.
We therefore have

\begin{cor}\label{cor:wtpoly}
$$
W^\Gamma_c(F,t)=\sum_{d\in\CP}W^\Gamma_c(F(d)^\circ,t).
$$
\end{cor} 

Now the equivariant Euler characteristic is an additive function, so the next statement is immediate.
\begin{cor}\label{lem:decomp1} We have
$$
\chi^\Gamma(F)=\sum_{d\in\CP}\chi^\Gamma(F(d)^\circ).
$$
\end{cor}


To prove a more explicit version of the last result, we shall require the following general fact.

\begin{prop}\label{prop:z}
Let $X$ be a $CW$-complex with a free $G$-action, where $G$ is a finite group. Assume that
the quotient space $Y=X/G$ has the homotopy type of a finite $CW$-complex. Then $X$ has the homotopy type of 
a finite $CW$-complex, and we have the following equation in the Grothendieck ring of $G$.
$$
\chi^G(X)=\chi(Y)\cdot \Reg_G,
$$
where $\Reg_G$ is the regular representation of $G$.
\end{prop}
\begin{proof}
 It follows from \cite{Za} (see also \cite[\S 3.13]{DLu}) that the virtual representation
$\chi^G(X)$ is an integer multiple  $c\cdot \Reg_G$ of the regular representation of $G$.
This may also be easily seen in our context because in view of the free
nature of the action of $G$ on $X$, the cohomology of $X$ is the cohomology of
a cochain complex such that $G$ acts in each degree as a multiple of the regular representation.
Thus this first assertion follows from the Hopf trace formula.

To determine $c$, take the inner product of   $\chi^G(X)=c\cdot \Reg_G$ with $1_G$.
Since $\dim H^i(X/G)=(H^i(X),1_G)_G$, we have $(\chi^G(X),1_G)_G=\chi(X/G)$.
But $(\Reg_G,1_G)_G=1$, whence $c=\chi(X/G)$. 
\end{proof}

The summands on the right of the expression in Lemma \ref{lem:decomp1} may now be described
more explicitly.

\begin{prop}\label{prop:fd1}
For each $d\in\CP$ let $g_d\in G$ be a $d$-regular element, i.e. an element
such that $E_d:=V(g_d,\zeta_d)\neq 0$. Then writing $R(d)$ for the cyclic group
$\langle(g_d,\zeta_d)\rangle$, we have
$$
\chi^\Gamma(F(d)\0)=c(d)\Ind_{R(d)}^\Gamma(1_{R(d)}),
$$
where $c(d)=\chi\left(F(d)^\circ/(G(d)\times\mu_m)\right)$.
\end{prop}
\begin{proof}
Note first that the intersections with $F$ of the maximal $\zeta_d$-eigenspaces
of elements of $G$ are disjoint. This is because if $E=V(g,\zeta_d)$ and $E'=V(g',\zeta_d)$
are two such eigenspaces, then for $v\in E\cap E'\cap F$, we have $gv=\zeta_d v=g' v$, whence 
by regularity, $g=g'$ and hence $E=E'$. It follows, since $\Gamma$ acts transitively on 
the maximal $\zeta_d$-eigenspaces, that if $E_d=V(g_d,\zeta_d)$ is one of these, 
and $\Gamma(d)$ is the stabiliser in
$\Gamma$ of $F(d)^\circ\cap E_d$, then 
\be\label{eq:ind1}
\chi^\Gamma(F(d)^\circ)=\Ind_{\Gamma(d)}^\Gamma(\chi^{\Gamma(d)}((F\cap E_d)^\circ)),
\ee
where $(F\cap E_d)^\circ=F(d)^\circ\cap E_d$.
Next let us determine $\Gamma(d)$. If $v\in E_d\cap F$ and $(g,\xi)\in\Gamma$, then
$(g,\xi)v=\xi\inv gv\in E_d$ implies that $g_dgv=\zeta_dgv=gg_dv$, whence by the regular nature of 
$v$, we have $g\in C_G(g_d)=G(d)$. The converse is easily checked, whence we see that
$\Gamma(d)=G(d)\times\mu_m$. 

Now $G(d)$ acts as a reflection group on $E_d$,  and it is proved in \cite{L96} and in \cite{DLo} that
the reflecting hyperplanes of this action are the intersections
with $E_d$ of those of $G$. Suppose $(g,\xi)\in\Gamma(d)$ and $v\in (F\cap E_d)^\circ$ are such that
$(g,\xi)v=v$. Then $gv=\xi v$.

It follows, since $v\in (F\cap E_d)^\circ$, that $\xi=\zeta_d^i$ for some $i$, and hence that
$g_d^iv=gv$, whence $(g,\xi)=(g_d,\zeta_d)^i\in \langle(g_d,\zeta_d)\rangle$. Conversely, each
element of the group $\langle(g_d,\zeta_d)\rangle$ fixes $F(d)^\circ\cap E_d$ pointwise. Hence
the quotient group $\Gamma(d)/\langle(g_d,\zeta_d)\rangle=\Gamma(d)/R(d)$ acts freely on $(F\cap E_d)^\circ$,
and it follows from Proposition \ref{prop:z} that
\be\label{eq:ind2}
\chi^{\Gamma(d)}((F\cap E_d)^\circ)=c(d)\Reg_{\Gamma(d)/\langle(g_d,\zeta_d)\rangle}
=c(d)\Ind_{R(d)}^{\Gamma(d)}(1_{R(d)}),
\ee
where $c(d)=\chi\left(F(d)^\circ/(G(d)\times\mu_m)\right)$, and $\Reg_H$ denotes the regular
 representation of a finite group $H$. 

Putting together \eqref{eq:ind1} and \eqref{eq:ind2} and using transitivity of induction, we obtain
the statement of the Proposition.
\end{proof}

The above proof yields the more general statement below, which applies to most of the invariants we wish to
investigate, such as the equivariant weight polynomials and the Hodge-Deligne polynomials.

\begin{sch}\label{sch:dec}
Let $\beta$ be an additive $\Gamma$-functor from constructible subspaces of $F$ to  the ring $S:=A\ot_\C R(\Gamma)$,
where $R(\Gamma)$ is the complex Grothendieck ring of $\Gamma$ and $A$ is any commutative ring.
 Then $S$ is naturally a $\Gamma$-module, and the statement that $\beta$ is a $\Gamma$-functor means
 that $\beta$ satisfies $\beta(\gamma(Y))=\gamma(\beta(Y))$ for all 
$\gamma\in\Gamma$ and constructible $Y\subseteq X$. 
Then 
\begin{enumerate}
\item $\beta(F)=\sum_{d\in\CP}\Ind_{\Gamma(d)}^{\Gamma}\beta((F\cap E_d)^0)$, where $\Gamma(d)=G(d)\times\mu_m$
and $E_d$ is any maximal $\zeta_d$-eigenspace of $V$.
\item $\Gamma(d)/R(d)$ acts freely on $(F\cap E_d)^0$, where $R(d)=\langle (g_d,\zeta_d)\rangle$ is the cyclic group
occurring in the proof of Proposition \ref{prop:fd1}.
\end{enumerate}
\end{sch}
\begin{proof}
From Lemma \ref{lem:decomp-gen} we have $\beta(F)=\sum_{d\in\CP}\beta(F(d)^0)$. But by the
argument at the beginning of the proof of Proposition \ref{prop:fd1}, $F(d)^0$ is
a disjoint union $F(d)^0=\amalg_E F(d)^0\cap E$, where the sum is over the distinct (maximal) $\zeta_d$-eigenspaces
of $G$ in $V$. It follows from the additivity property of $\beta$ that $\beta(F(d)^0=\sum_E\beta(F(d)^0\cap E)$.
Moreover, since the eigenspaces $E$ are conjugate under $G$, and hence {\em a fortiori} under $\Gamma$, it follows
from the defining property of $\beta$ that the subspaces  $\beta(F(d)^0\cap E)$ of the $R(\Gamma)$-module $\beta(F(f)^0)$ are  
permuted transitively by $\Gamma$; further, again by the argument in the proof of  Proposition \ref{prop:fd1},
the stabilizer of one of these subspaces is $\Gamma(d)$. It follows from the functorial property and additivity
that $\beta(F)=\Ind_{\Gamma(d)}^\Gamma(\beta(F(d)\cap E_d)^0)$, where $E_d$ is a particular maximal $\zeta_d$-eigenspace.

This proves (1), while (2) is proved in the proof above.
\end{proof}

Examples of functors $\beta$ to which Scholium \ref{sch:dec} may be applied include the Euler characteristic,
the compactly supported weight polynomial, and the compactly supported Hodge-Deligne polynomial
(see \eqref{eq:hd} in \S\ref{s:mh} below). The Euler characteristic case is the easiest to handle. 


Let $U=\PP(M):=M/\C^{\times}$ and consider the map $p:F\lr U$ given by $p(v)=[v]$. This is an unramified 
$\mu_m$-covering, and $U$ may be identified with $F/\mu_m$ (see \ref{eq:diag} below). Let $U(d),U(d)^\circ$ be the analogues for the pair
$E_d,G(d)$ of the spaces $U,U^\circ$ for $V,G$. Thus in particular, $U(d)^\circ=p(E_d^\circ)$. Then evidently
$F(d)^\circ/(G(d)\times\mu_m)$ may be identified with $U(d)^\circ/G(d)$. 



Using this notation, Lemma \ref{lem:decomp1} and Proposition \ref{prop:fd1} may be combined 
to yield the following statement.

\begin{cor}\label{cor:euler}(cf. \cite[Theorem 3.13]{DeL})
Let $G$ be an irreducible unitary reflection group acting on $V=\C^\ell$. Let $M,Q,m,F,U,\CP$ etc. be as in the 
discussion above, and for $d\in\CP$ let $R(d)$ be the cyclic subgroup of $\Gamma=G\times\mu_m$ 
defined in Definition \ref{def:rd}. Then
\be\label{eq:thm}
\chi^{\Gamma}(F)=\sum_{d\in\CP}\chi(U(d)\0/G(d))\Ind_{R(d)}^\Gamma(1).
\ee
\end{cor}
\begin{rk}\label{rem:valchi}
The value of $\chi(U\0/G)$ has been computed case by case in \cite[Theorem 3.15]{DeL}. 
We shall give an essentially case free proof below. The key result for the examples below is 
that for $G=G(r,1,\ell)$ (an imprimitive reflection group) $\chi(U\0/G)=0$ if $\ell>2$ or
if $\ell= 2$ and $r=1$. If $\ell=2$ and $r>1$ then $\chi(U\0/G)=-1$; otherwise (if $\ell=1$) $\chi(U\0/G)=1$.
\end{rk}

\subsection{A diagram of spaces} Maintain the notation above, and write 
$\wt M:=\{(v,\xi)\in M\times\C^\times\mid Q(v)=\xi^m\}$. Then we have maps $p_1:\wt M\to M$
and $\pi_1:\wt M\to F$, where $p_1$ is the first projection, and $\pi_1(\xi,v)=\xi\inv v$.
We also have $\pi: M\to U=\PP(M)$ and $p:M\to U$ given by $\pi(v)=[v]$ and $p(v)=[v]$, the latter for 
$v\in F$.
These maps fit together in a commutative diagram as follows.
\begin{equation}\label{eq:diag}
\begin{CD}
\wt M @>p_1>> M\\
@V\pi_1VV @VV\pi V\\
F @>p>> U=\PP(M)\\
\end{CD}\end{equation}

The group $\Gamma$ acts on $\wt M$ component-wise: $(g,\zeta)(v,\xi):=(gv,\zeta\xi)$ for
$(g,\zeta)\in \Gamma$ and $(v,\xi)\in\wt M$. Since $\pi_1(v,\xi)=\xi\inv v$, it is easily checked that
$\pi_1$ respects the $\Gamma$-action, and that the horizontal arrows are maps to the quotient 
by $\mu_m$.

Note that the map $(v,\xi)\mapsto (\xi\inv v,\xi):\wt M\to F\times\C^\times$ is an isomorphism of varieties,
and the the map $p_1:F\times\C^\times\to M$ is given by $(v,\xi)\mapsto \xi v$. With this identification,
the $\Gamma$-action on $F\times \C^\times$ is given by $(g,\zeta).(w,\xi)=(\zeta\inv gw,\zeta\xi)$,
for $(g,\zeta)\in\Gamma$ and $(w,\xi)\in F\times\C^\times$. In particular, this realises $M$ as the
quotient of $F\times\C^\times$ by $\mu_m$, the latter acting via $\zeta(w,\xi)=(\zeta\inv w,\zeta\xi)$.

Notice that the action of both $G$ and of $\mu_m$ on $F$ is free. We next record the result of applying Proposition \ref{prop:z} 
to these free actions.

\begin{prop}\label{prop:ffree}
\begin{enumerate}
\item We have $\chi^G(F)=\chi(F/G)\Reg_G.$
\item We have $\chi^{\mu_m}(F)=\chi(U)\Reg_{\mu_m}$.
\item Let $1=m_1^*\leq m_2^*\leq\dots\leq m_r^*$ be the non-zero coexponents of $G$ (see \cite[Def. 10.27, p.257]{LT}).
Here $r=\codim_V(\cap_{H\in\CA_G}H)$ is the rank of $G$.
Then $\chi(U)=(-1)^{r-1}(m_2^*-1)(m_3^*-1)\dots(m_{r}^*-1)$.
\item We have $\chi(F/G)=|G|\inv m (-1)^{r-1}(m_2^*-1)(m_3^*-1)\dots(m_{r}^*-1)$.
\end{enumerate}
\end{prop}
\begin{proof}
The assertions (1) and (2) are immediate from Proposition \ref{prop:z}. To see (3), note that by
\cite[p. 257]{LT}, the Poincar\'e polynomial of $M$ is $P_M(t)=\prod_{i=1}^\ell(1+m_i^*t)=\prod_{i=1}^r(1+m_i^*t)$, where the $m_i^*$
are the coexponents of $G$. Moreover precisely $r$ of the coexponents are non-zero, and the smallest of these
is the degree of the Euler form, viz. $m_1^*=1$. But from
the diagram above, we see that $P_U(t)=\frac{P_M(t)}{(1+t)}$. This proves (3).
As for (4), it follows from (1) and (2) by evaluation at the identity element of the relevant group,
 that $\chi(F)=m\chi(U)=|G|\chi(F/G)$. Hence (4) is immediate from
(3).
\end{proof}
\begin{rk}
Since $\chi(F/G)\in\Z$, Proposition \ref{prop:ffree} implies the following divisibility result for any unitary reflection group
$G$. Recall (cf. \cite[(1.3)]{L05}) that $m=n_G+N_G$, where $n_G=\sum_im_i=\sum_{H\in\CA_H}(e_H-1)$
and $N_G=\sum_im_i^*=|\CA_G|$, and the $m_i$ and $m_i^*$ are respectively
the exponents and coexponents of $G$. Since $|G|=\prod_i(1+m_i)$, it follows that
$$
\prod_{i=1}^\ell(1+m_i)\text{ divides }\left(\sum_{i=1}^\ell( m_i+m_i^*)\right)\prod_{i=2}^\ell(1-m_i^*).
$$
\end{rk}

\begin{ex}\label{ex:eulsym}
Let $G$ be the symmetric group $\Sym_{\ell+1}$ acting on $\C^{\ell+1}$ by permutation of coordinates.
The non-zero coexponents in this case are $1,2,\dots,\ell$, and $m=\ell(\ell+1)$. Hence in this case
we have $\chi(U)=(-1)^{\ell-1}(\ell-1)!$, $\chi(F/G)=(-1)^{\ell+1}$ and $\chi(F)=(-1)^{\ell+1}(\ell+1)!$. 
\end{ex}

We finish this section with a closed
(but finitely recursive) formula for $\chi(U^0/G)$, which may be applied to give
a case free proof of \cite[Theorem 3.15]{DeL}.
\begin{thm}\label{thm:dl-chi}
Let $G$ be a finite reflection group in the complex vector space $V$, $M$ the corresponding hyperplane
complement, $U=\PP(M)$, etc, as above. Let $\CP$ be the poset of integers $d$ such that $d$ is the 
$\gcd$ of the degrees of $G$ which it divides, where $e\leq d$ in $\CP$ if $d|e$. Then in the notation above,
$$
\chi(U^0/G)=|Z(G)|\sum_{d\in\CP}\mu(d)|G(d)|\inv\prod_{i\geq 2}(1-m_i^*(d)),
$$
where $\mu(d)=\mu(d,|Z(G)|)$ is the M\"obius function of the poset $\CP$, $r(d)$ is the rank of the
reflection group $G(d)$, and the $m_i^*(d)$ are the coexponents of the reflection group $G(d)$
written so that $m_1^*(d)\leq m_2^*(d)\leq\dots$.
\end{thm}
\begin{proof}
It follows from Proposition \ref{prop:ffree} that the multiplicity of the trivial character $\gamma_0$
of the monodromy group $\mu_m$ in $\chi(F)$ is $\chi(U)$. Hence applying Corollary \ref{cor:euler},
\be\label{eq:c1}
\chi(U)=\sum_{d\in\CP}\chi(U(d)\0/G(d))(\Ind_{R(d)}^\Gamma(1),\gamma_0)_{\mu_m},
\ee
where $(\;,\;)_{\mu_m}$ denotes the usual inner product of characters. But an easy application of
Mackey's formula shows that $(\Res^\Gamma_{\mu_m}\Ind_{R(d)}^\Gamma(1),\gamma_0)_{\mu_m}=\frac{|G|}{d}$.
It follows from \eqref{eq:c1} that
\be\label{eq:c2}
\chi(U)=(\chi^\Gamma(F),\gamma_0)_{\mu_m}=\sum_{d\in\CP}\chi(U(d)\0/G(d))\frac{|G|}{d}.
\ee
If we define, for any reflection group $G$, $c(G)=\frac{\chi(U)}{|G|}$ and $c_0(G)=\chi(U^0/G)$,
then \eqref{eq:c2} may be rearranged to read 
\be\label{eq:c3}
c(G)=\sum_{d\in\CP(G)}\frac{c_0(G(d))}{d},
\ee 
where $\CP(G)$ indicates the poset $\CP$ for the group $G$.

However if $e\leq d$ in $\CP(G)$, it is shown in \cite{LSp} that
\be\label{eq:c4}
G(d)(e)=G(e).
\ee

Moreover using the fact that the degrees of $G(d)$ are precisely those degrees of $G$ which are divisible by $d$,
a short calculation verifies that
\be\label{eq:c5}
\CP(G(d))=\{e\in\CP(G)\mid e\leq d\}.
\ee

It follows from \eqref{eq:c4} and \eqref{eq:c5} that we may apply  \eqref{eq:c2} to the group $G(d)$ to obtain,
for any $d\in\CP$
\be\label{eq:c6}
c(G(d))=\sum_{e\in\CP(G),e\leq d}\frac{c_0(G(e))}{e}.
\ee 
We may now invert the relation \eqref{eq:c6} using the M\"obius function $\mu(e,d)$ of $\CP$ to obtain for any $d\in\CP$,
\be\label{eq:c7}
\frac{c_0(G(d))}{d}=\sum_{e\in\CP(G),e\leq d}\mu(e,d)c(G(e)).
\ee 
Taking into account that the top element of $\CP$ is $|Z(G)|$ (which is the $\gcd$ of the degrees of $G$),
the stated relation is simply the case $d=|Z(G)|$ of \eqref{eq:c7}, where we write $\mu(d)=\mu(d,|Z(G)|)$,
and use the formula in Proposition \ref{prop:ffree}(3) for $\chi(U)$. 
\end{proof}

\subsection{A factorisation result}\label{ssec:factor}
Our ultimate objective is to determine the $\Gamma$-module 
structure of $H^i(F,\C)$ for each $i$, or equivalently, to determine the equivariant Poincar\'e
polynomial 
\be\label{eq:defpoin}
P^\Gamma(F,t):=\sum_{i\geq 0}H^i(F,\C)t^i\in R_+(\Gamma)[t],
\ee
where $R_+(\Gamma)$ is the multiplicative submonoid of actual representations of $\Gamma$ in its 
Grothendieck ring. Note that $P^\Gamma(F,-1)=\chi^{\Gamma}(F)$.

For any finite group $H$, denote by $I(H)$ the set of irreducible $\C$-representations of $H$.
Then $R_+(\Gamma)=\oplus_{\theta\in I(\Gamma)}\Z_{\geq 0}\theta$.

Since $\Gamma=G\times\mu_m$, every element of $I(\Gamma)$ is of the form $\rho\ot \gamma$, where
$\rho\in I(G)$, and $\gamma\in I(\mu_m)$. We shall henceforth use the following notation for
elements of $I(\mu_m)$: define $\gamma_1\in I(\mu_m)$ by $\gamma_1(\zeta_m)=\zeta_m$ (recall
that for any $d$, $\zeta_d=\exp(\frac{2\pi\sqrt{-1}}{d})$); then for $i=0,1,\dots,m-1$
define $\gamma_i=\gamma_1^i$.

We shall now focus on the case where $e_H=2$ for each hyperplane $H$. This includes the real Coxeter groups,
but also many other cases, such as the exceptional groups $G_{12},G_{13},G_{22},G_{24},G_{27},G_{28},G_{29},
G_{31},G_{33}$ and $G_{34}$ (see \cite[Tables D1, D2]{LT}). 
In this case $Q=\prod_{H\in\CA}\ell_H^2=Q_0^2$, where $Q_0=\prod_{H\in\CA}\ell_H$. Clearly
$F=\{v\in V\mid Q(v)=1\}=\{v\in V\mid Q_0(v)=\pm 1\}=F_0\amalg F_-$, where $F_{-}=\{v\in V\mid Q_0(v)=- 1\}$,
and as in the introduction, $F_0=\{v\in V\mid Q_0(v)=1\}$.

Now it is known \cite[\S 9.4]{LT} that there is a character $\ve$ of $G$, the `alternating character', 
such that $Q_0(g\inv v)=\ve(g)Q_0(v)$ for all $g\in G$ and $v\in V$; in fact $\ve(g)=\det_V(g)$.
Observe that in this case $m$ is always even.

We shall prove the following factorisation for $P^\Gamma(F,t)$.

\begin{prop}\label{prop:factor} Let $G$ be a reflection group such that $e_H=2$ for each reflecting hyperplane
$H$, and maintaining the above notation, write $\Gamma_0:=\Ker(\ve\ot\gamma_{\frac{m}{2}})$. Then
$$
P^\Gamma(F,t)=(1\ot\gamma_0 + \ve\ot\gamma_{\frac{m}{2}}) P_0^\Gamma(F,t),
$$
where $P_0^\Gamma(F,t)$ is an element of $R_+(\Gamma,t)$ whose restriction to $\Gamma_0$
is $P^{\Gamma_0}(F_0,t)$. 

In particular,
$$
\chi^{\Gamma}(F)=(1\ot\gamma_0 + \ve\ot\gamma_{\frac{m}{2}})\chi_0^{\Gamma}(F)
$$
where $\chi_0^\Gamma(F)$ is an element of $R(\Gamma)$ whose restriction to $\Gamma_0$
is $\chi^{\Gamma_0}(F_0)$.
\end{prop}
\begin{proof}
It is evident from the above discussion that for each $i$, $H^i(F,\C)=H^i(F_0,\C)\oplus
H^i(F_-,\C)$; moreover, since $\Gamma$ permutes the spaces $F_{0},F_-$ it permutes the
two summands above. It follows that if $\Gamma_0$ is the stabiliser in $\Gamma$ of the variety
$F_0$, then the following holds in $R(\Gamma)$.
\be\label{eq:induced}
H^i(F,\C)=\Ind_{\Gamma_0}^\Gamma(H^i(F_0,\C)).
\ee

Next we identify $\Gamma_0$. The element $(g,\xi)\in\Gamma$ fixes $F_0$ if and only if 
it fixes $Q_0$. But $(g,\xi)Q_0(v)=\xi^{\frac{m}{2}}\ve(g)Q_0(v)$. Hence $\Gamma_0
=\Ker(\ve\ot\gamma_{\frac{m}{2}})$.

Now to understand induction from $\Gamma_0$ to $\Gamma$, we use Clifford theory, applied to
our simple case (\cite[Theorem 1]{GK}, \cite{L72}). The facts we require are as follows. 
If $H$ is a finite group and $K$ is a normal subgroup, 
we identify $I(H/K)$ with the subset of $I(H)$ consisting of representations in which $K$
acts trivially. For $\theta\in I(H)$, define $F(\theta)=\{\xi\in I(H/K)\mid \theta\xi=\theta\}$. 
Then $\Res^H_K(\theta)$ is irreducible if and only if $F(\theta)=1$.

In our present situation, taking $H=\Gamma$ and $K=\Gamma_0$, 
any element of $I(\Gamma)$ is of the form $\theta=\rho\ot\gamma_i$, where $\rho\in I(G)$ and 
$\gamma_i$ is as defined above. In particular, $I(\Gamma/\Gamma_0)$
may be identified with $\{1\ot\gamma_0, \ve\ot\gamma_{\frac{m}{2}}\}$.

Now for any $\theta=\rho\ot\gamma_i\in I(\Gamma)$ we have $\ve\ot\gamma_{\frac{m}{2}}\theta=
\ve\rho\ot\gamma_{\frac{m}{2}+i}\neq\theta$, since the second factor is distinct from $\gamma_i$.
Thus from Clifford theory we deduce that each element $\theta\in I(\Gamma)$ restricts to
an irreducible representation of $\Gamma_0$. It follows easily by Frobenius reciprocity
that for any element $\chi\in I(\Gamma_0)$, $\Ind_{\Gamma_0}^\Gamma(\chi)=
(1\ot\gamma_0 + \ve\ot\gamma_{\frac{m}{2}})\beta$ for some $\beta\in I(\Gamma)$,
whose restriction to $\Gamma_0$ is $\chi$. Since the last statement 
is linear in $\chi$, the same thing holds for any element $\chi\in R_+(\Gamma_0)$,
and applying this statement to $H^i(F_0,\C)\in R_+(\Gamma_0)$, using \eqref{eq:induced},
we obtain the stated result.
\end{proof}

\begin{rk}\label{rem:irrplus}
It follows from the above proof that each irreducible representation $\sigma\in I(\Gamma_0)$ is the 
restriction to $\Gamma_0$ of precisely two elements of $I(\Gamma)$, which may be written
$\rho\ot\gamma_i$ and $\ve\rho\ot\gamma_{i+\frac{m}{2}}$. It follows that 
\be
\begin{aligned}
&\text{ Each element of $I(\Gamma_0)$ is the restriction to $\Gamma_0$ of a unique}\\
&\text{ representation
of the form $\rho\ot\gamma_i$, where $0\leq i\leq\frac{m}{2}-1$.}\\
\end{aligned}
\ee
We shall therefore abuse notation by writing the elements of $I(\Gamma_0)$ as $\rho\ot\gamma_i$,
with $\rho\in I(G)$ and $0\leq i\leq\frac{m}{2}-1$.
\end{rk}

\subsection{Quotient by $G$}
It is easy to deduce from the formula \eqref{eq:thm} the $\mu_m$-equivariant Euler characteristic
of $F/G$. The result is as follows.

\begin{prop}\label{prop:muequi}
Let $\gamma_0,\gamma_1,\dots,\gamma_{m-1}$ be the characters of $\mu_m$, where 
$\gamma_i(\zeta_m)=\zeta_m^i$. Then
\be
\chi^{\mu_m}(F/G)=\sum_{d\in\CP}\chi(U(d)\0/G(d))\sum_{j=0}^{\frac{m}{d}-1}\gamma_{dj}.
\ee
\end{prop}
\begin{proof}
Clearly $\chi^{\mu_m}(F/G)$ is the $1_G$-isotypic part of the virtual representation on the 
right side of \eqref{eq:thm}. To compute this, note first that for any representation $\alpha$ of $\Gamma$,
we have by Frobenius reciprocity
\be\label{eq:frob}
(\Ind_{R(d)}^\Gamma(1),\alpha)_\Gamma=(1_{R(d)},\Res^\Gamma_{R(d)}(\alpha))_{R(d)}.
\ee
It is therefore evident that 
$$
(\Ind_{R(d)}^\Gamma(1),1_G\otimes\gamma_i)_\Gamma=
\begin{cases}
1 \text{ if } \gamma_i|_{\mu_d}=1\\
0\text{ otherwise}.
\end{cases}
$$
The stated formula follows easily.
\end{proof}

\begin{rk}\label{rem:quot}
Let $f_1,\dots,f_\ell$ be a set of basic invariants for $G$. The orbit map $V\lr V/G$ may be realised
as $\pi:x\mapsto (f_1(x),\dots,f_\ell(x))$; since $Q$ is $G$-invariant, there is a unique polynomial
$\Delta(y_1,\dots,y_\ell)$ such that $Q(x)=\Delta(f_1(x),\dots,f_\ell(x))$. The Milnor fibre $F$
is mapped by $\pi$ to $F/G=F_\Delta:=\{y\mid\Delta(y)=1\}$, and the monodromy $\mu_m$ 
acts on $F_\Delta$ via $y\mapsto (\zeta^{d_1}y_1,\dots,\zeta^{d_\ell}y_\ell)$, $(\zeta\in\mu_m)$.
For the symmetric group $S_{\ell+1}$, $\Delta(y)$ is just the usual discriminant of the polynomial
$t^{\ell+1}-y_1t^\ell+y_2t^{\ell-1}-\dots+(-1)^{\ell+1}y_{\ell+1}$. In view of the remark at the beginning of Example \ref{ex:typea} below,
one may take $y_1=0$ in this case, i.e. it suffices to consider the Milnor fibre of the polynomial
$\Delta_0(y_2,\dots,y_{\ell+1}):=\Delta(0,y_2,\dots,y_{\ell+1})$.
\end{rk}

\begin{ex}\label{ex:typea}
We take $G=\Sym_{\ell+1}$ acting by permutation of coordinates on $\C^{\ell+1}$. 
This action is irreducible on the hyperplane $x_1+\dots+x_{\ell+1}=0$, and if $F_0$ is
the Milnor fibre corresponding to this irreducible action, we have $F=F_0\times\C$, so that
the $\Gamma$-modules $H^*(F_0)$ and $H^*(F)$ are isomorphic. We consider
here the graded space $H^*(F)$. The regular
numbers are all integers $d$ which divide $\ell$ or $\ell+1$. The degrees of $G$ are $1,2,\dots,\ell+1$.
For any regular $d$, since $d$ is one of the degrees, $\bar d=d$, and so $\CP$ coincides with
the set of regular numbers.

Let $d\in\CP$. It is easily verified that $G(d)\simeq G(d,1,[\frac{\ell+1}{d}])$.
Hence in the notation above, by Remark \ref{rem:valchi}
$\chi(U(d)\0/G(d))=0$ unless $[\frac{\ell+1}{d}]\leq 2$, that is unless $d=\ell,\ell+1,
\frac{\ell+1}{2}\text{ or } \frac{\ell}{2}$. The corresponding values of $\chi(U(d)\0/G(d))$
in these cases are respectively $1,1,-1,-1$.

Thus the main results above may be written in this case as follows.

\be\label{eq:thm-sym}
\chi^{\Gamma}(F)=
\begin{cases}
\Ind_{R(\ell)}^\Gamma(1)+\Ind_{R(\ell+1)}^\Gamma(1)-\Ind_{R(\frac{\ell}{2})}^\Gamma(1)
\text { if $\ell$ is even}\\
\Ind_{R(\ell)}^\Gamma(1)+\Ind_{R(\ell+1)}^\Gamma(1)-\Ind_{R(\frac{\ell+1}{2})}^\Gamma(1)
\text { if $\ell$ is odd}\\
\end{cases}
\ee
where $R(d)$ is the cyclic group (of order $d$) generated by $(g_d,\zeta_d)$, where $g_d$ is a product
of $[\frac{\ell+1}{d}]$ $d$-cycles (see Definition \ref{def:rd}).

For the $1_{\Sym_{\ell+1}}$-isotypic part, we have
\be\label{eq:iso-sym}
\chi^{\mu_m}(F/G)=
\begin{cases}
\sum_{j=0}^{\ell}\gamma_{\ell j}+\sum_{j=0}^{\ell-1}\gamma_{(\ell+1) j}-\sum_{j=0}^{2\ell+1}
\gamma_{\frac{\ell}{2}j}\text { if $\ell$ is even}\\
\sum_{j=0}^{\ell}\gamma_{\ell j}+\sum_{j=0}^{\ell-1}\gamma_{(\ell+1) j}-\sum_{j=0}^{2\ell-1}
\gamma_{\frac{(\ell+1) }{2}j}\text { if $\ell$ is odd}.\\
\end{cases}
\ee

These formulae may be written more transparently as

\be\label{eq:iso-sym2}
\chi^{\mu_m}(F/G)=
\begin{cases}
\sum_{j=0}^{\ell-1}\gamma_{(\ell+1) j}-\sum_{j=0}^{\ell}
\gamma_{\frac{\ell}{2}(2j+1)}\text { if $\ell$ is even}\\
\sum_{j=0}^{\ell}\gamma_{\ell j}-\sum_{j=0}^{\ell-1}
\gamma_{\frac{(\ell+1)}{2}(2j+1)}\text { if $\ell$ is odd}.\\
\end{cases}
\ee
\end{ex}
Let us check these formulae for compatibility with the results of \cite{DPS}.
It is shown in \cite[Theorem 4]{Ca} that if $B_{\ell+1}$ is the braid group on
$\ell+1$ strings, and $R=\C[q,q\inv]$ is the $B_{\ell+1}$-module upon which the generators
of $B_{\ell+1}$ act as multiplication by $-q$, then 
for each degree $k$, we have
$$
H^k(F/G,\C)\simeq H^{k+1}(B_{\ell+1},R).
$$
Given this, the main result \cite[Theorem, p. 739]{DPS}, may be stated in the language of
the above exposition as follows. 

\begin{prop}\label{prop:dps} For any character $\gamma\in\widehat{\mu_m}$ write $|\gamma|$
for its order. For any integer $h\geq 2$ which divides $\ell$ or $\ell+1$, we write 
$i(h)=\left[\frac{\ell+1}{h}\right]$, and $k(h)=i(h)(h-2)$. 
If $k\neq k(h)$ for some $h$, then $H^k(F/G,\C)=0$.
For the remaining cases we have the following equation in the Grothendieck ring of $\mu_m$.
$$
H^{k(h)}(F/G,\C)=
\begin{cases}
\sum_{\gamma\in\wh{\mu_m},|\gamma|=2h}\gamma\text{ if $h$ is odd},\\
\sum_{\gamma\in\wh{\mu_m},|\gamma|=\frac{h}{2}}\gamma\text{ if }h\equiv 2(\text{mod }4),\\
\sum_{\gamma\in\wh{\mu_m},|\gamma|={h}}\gamma\text{ if }h\equiv 0(\text{mod }4).\\
\end{cases}
$$
\end{prop}

The above Proposition may be applied to deduce information about $H^*(F_0,\C)$ as follows.
Note that $H^*(F_0,\C)$ is {\it a priori} a module for $\mu_{\frac{m}{2}}$. The next result 
provides some information concerning this action.

\begin{cor}\label{cor:dps} Let $G$ etc. be as in Proposition \ref{prop:factor}; thus we have
$Q_0$, $\Gamma_0$, etc. Then $F/G\simeq F_0/G_0$, where $G_0=\Ker(\ve)$, and the 
$\mu_{\frac{m}{2}}$-action on $H^i(F_0/G_0,\C)$ is just the restriction to $\mu_{\frac{m}{2}}$
of the $\mu_m$-action on $H^i(F/G,\C)$.
\end{cor}
\begin{proof}
We have $F=F_0\amalg F_-$, and $G$ permutes the two pieces; that is, for each $g\in G$, 
$g F_0=F_0$ or $g F_0=F_-$. Moreover there is an element of $G$ which interchanges the two pieces.
It follows that $F/G$ may be identified with the equivalence classes of elements of $F_0$,
equivalence being defined as $x\equiv y$ if $y=gx$ for some $g\in G$. But if $x,y\in F_0$,
and $gx=y$, then $g$ must be in $G_0$, which is the stabiliser of $F_0$, and so $F/G\simeq F_0/G_0$.
\end{proof}

\begin{rk}\label{rem:isotypic}
Let $d=\frac{m}{2}=\deg Q_0$. Then $\mu_d=\Gamma_0\cap\mu_m$, and $\mu_m$ is  a normal subgroup of
$\Gamma$, which is contained in the centre of $\Gamma$. Hence $\Gamma_0/\mu_d\cong\Gamma/\mu_m\cong G$.
It follows that if $\gamma\in\wh{\mu_d}$ and $M$ is any $\Gamma_0$-module, then the $\gamma$-isotypic component
$M^{\gamma}$ of $M$ is a module for $G=\Gamma_0/\mu_d$. We shall use this repeatedly in \S 4 below.
Note that $\dim M^\gamma=(M,\gamma)_{\mu_d}$.
\end{rk}

The next corollary of Proposition \ref{prop:dps} provides a strengthening of the formula \eqref{eq:iso-sym2},
in that it asserts that there is no cancellation in arriving at that formula.

\begin{cor}\label{cor:hevenodd}
Let $G=\Sym_{\ell+1}$ as in example \ref{ex:typea}. Then
$$
\sum_{i\equiv\ell(\mod 2)}H^{i,\mu_m}(F/G,\C)=
\begin{cases}
\sum_{i=0}^{\ell-1}\gamma_{(\ell+1)i}\text { if }\ell\equiv 0(\mod 2)\\
\sum_{i=0}^{\ell-1}\gamma_{\frac{\ell+1}{2}(2i+1)}\text { if }\ell\equiv 1(\mod 2),\\
\end{cases}
$$
and
$$
\sum_{i\equiv\ell+1(\mod 2)}H^{i,\mu_m}(F/G,\C)=
\begin{cases}
\sum_{i=0}^{\ell}\gamma_{\ell i}\text { if }\ell+1\equiv 0(\mod 2)\\
\sum_{i=0}^{\ell}\gamma_{\frac{\ell}{2}(2i+1)}\text { if }\ell+1\equiv 1(\mod 2).\\
\end{cases}
$$
\end{cor}
\begin{proof}
We use the notation of Proposition \ref{prop:dps}. First observe that a short computation shows that
\be\label{eq:dps1}
\begin{aligned}
\text{ If }&h|\ell, \text{ then }k(h)\equiv\ell(\mod 2),\text { and }\\
\text{ if }&h|\ell+1, \text{ then }k(h)\equiv\ell+1(\mod 2).\\
\end{aligned}
\ee
A further calculation using the details in Proposition \ref{prop:dps} now proves that
if $\ell$ is even, then 
$$
\sum_{i\equiv\ell(\mod 2)}H^{i,\mu_m}(F/G,\C)=\sum_{|\gamma||\ell}\gamma,
$$
while if $\ell$ is odd, then
$$
\sum_{i\equiv\ell(\mod 2)}H^{i,\mu_m}(F/G,\C)=\sum_{|\gamma|=2b\text{ where }b|\ell}\gamma.
$$
Putting these facts together with the analogous ones for $\ell+1$, one obtains the Corollary.
\end{proof}

\begin{rk}\label{rem:nocanc}
Corollary \ref{cor:hevenodd} both confirms the formula \eqref{eq:iso-sym2}, and shows further that there is no cancellation
in arriving at that formula. It is useful to reformulate it as follows.
\end{rk}

\begin{cor}\label{cor:dps} Let $G=\Sym_{\ell+1}$ as above.
For any integer $i$ such that $0\leq i\leq \ell-1$, define a subset $\CM_i\subseteq\wh{\mu_m}$ as follows.
Let
$$
h(i):=\begin{cases}
\frac{2\ell}{\ell-i}\text{ if }i\equiv\ell(\mod 2)\\
\frac{2(\ell+1)}{\ell+1-i}\text{ if }i\equiv\ell+1(\mod 2).\\
\end{cases}
$$
Define
$$
\CM_i=\begin{cases}
\{\gamma\in\wh{\mu_m}\mid |\gamma|=h(i)\}\text{ if }h(i)\in\Z\text{ and }h(i)\equiv 0(\mod 4)\\
\{\gamma\in\wh{\mu_m}\mid |\gamma|=\frac{h(i)}{2}\}\text{ if }h(i)\in\Z\text{ and }h(i)\equiv 2(\mod 4)\\
\{\gamma\in\wh{\mu_m}\mid |\gamma|=2h(i)\}\text{ if }h(i)\in\Z\text{ and }h(i)\equiv 1\text{ or }3(\mod 4)\\
\emptyset\text{ otherwise.}\\
\end{cases}
$$
The $1_G$-isotypic part of $P^\Gamma(F,t)$ (see \eqref{eq:defpoin}) is given by
$$
P^{\mu_m}(F/G,t)=P^\Gamma(F,t)^{1_G}=\sum_{i=0}^{\ell-1}\sum_{\gamma\in\CM_i}(1_G\ot\gamma) t^i.
$$
\end{cor}

This is an easy consequence of Proposition \ref{prop:dps} and Corollary \ref{cor:hevenodd}.

\begin{rk}\label{rem:dps-use}
The way in which Corollary \ref{cor:dps} may be used is as follows.
For elements $A,B\in R(\Gamma)[t]$
say that $A\geq B$ if $A-B\in R_+(\Gamma)[t]$.
Then $P^\Gamma(F,t)\geq 1_G\ot P^{\mu_m}(F/G,t)$, and 
$\left(P^\Gamma(F,t)- 1_G\ot P^{\mu_m}(F/G,t)\right)^{1_G}=0$.
\end{rk}

\begin{ex}\label{ex:typea2} Let $G=\Sym_3$. Then $m=6$ and we 
denote by $1$, $\ve$ and $\rho$ respectively the trivial one dimensional representation,
the sign representation and the (irreducible, two dimensional) reflection representation of $G$. 
The cohomology of $F$ in this case is completely determined by the formula \eqref{eq:thm-sym}
since there is no cancellation when the Euler characteristic is taken.
Using the notation above, the cohomology of Milnor fibre $F$ in this case has the 
following description as $\Gamma$-representation.

\medskip

(i) $H^0(F)=1\ot\gamma_0 + \ve\ot\gamma_3$.

\medskip

(ii) $H^1(F)=1\ot\gamma_1 + 1\ot\gamma_5 + \ve\ot\gamma_2 + \ve\ot\gamma_4 + 
\rho\ot\gamma_0 + \rho\ot\gamma_3$

\medskip

          $=(1\ot\gamma_0 + \ve\ot\gamma_3)(1\ot\gamma_1 + 1\ot\gamma_5 + \rho\ot\gamma_0)$.
   
Thus in particular, in the notation of Proposition \ref{prop:factor} 
$$
P^\Gamma(F,t)=(1\ot\gamma_0 + \ve\ot\gamma_3)\left( 1\ot\gamma_0 +(1\ot(\gamma_1 + \gamma_5) + 
\rho\ot\gamma_0)t\right).
$$

Note that this shows that $P^{\Gamma_0}(F_0,t)=1\ot\gamma_0+(1\ot(\gamma_1 +\gamma_5) +
\rho\ot\gamma_0)t$, which is consistent with \cite[Table 2]{Se}. Note that we also have, consistent 
with the notation of Remark \ref{rem:irrplus}, 
$P^{\Gamma_0}(F_0,t)=1\ot\gamma_0+(1\ot\gamma_1 +\ve\ot\gamma_2 +
\rho\ot\gamma_0)t$.

We note further that $H^1(F/G)=H^1(F)^G=\gamma_1+\gamma_5$, which agrees with the known monodromy 
of the cusp $F_{\Delta_0}:\Delta_0(y_2,y_3)=0$ (see Remark \ref{rem:quot}).

\end{ex}


\section{Groups of rank two.}\label{sec:rank2}

In this section we shall determine the the polynomials $P^\Gamma(F,t)$ and $W_c^\Gamma(F,t)$ for all `uniform' unitary
reflection groups $G$ of rank two, i.e. those which satisfy the condition \eqref{eq:cond} below, using the fact that in this case, 
the polynomial $Q_0:=\prod_{H\in\CA_G}\ell_H$ has an isolated singularity at $0$, and so the theory of \cite{OS1,OS2} applies.
For the other groups of rank $2$, one may apply an Euler characteristic argument directly.

\subsection{Reduction to the reduced case}
Say that the unitary reflection group $G$ is {\it uniform} if it satisfies the following
condition (for notation see the first paragraph of \S \ref{sec:intro}).
\be\label{eq:cond}
\text{For all }H\in\CA_G, e_H=|G_H|:=e\text{ is independent of }H.
\ee 

We write $d=\deg(Q_0)=|\CA_G|$ and $m=\deg(Q)$, where $Q=Q_0^e=\prod_{H\in\CA_G}\ell_H^{e_H}$,
so that $m=de$. A list of the irreducible uniform groups in dimension two is given in Table 1.

\begin{center}\label{table 1}
\begin{tabular}{|l|c|c|c|}
\hline
{Group}& {$e$}&$d$ & degrees of  $G$\\
\hline 
\hline
$G(2p,p,2)\;(p\geq 2)$& 2  &$2p+2$ & $2p,4$ (imprimitive)\\
\hline
$G(p,p,2)$ & $2$ &$p$ & $2,p$ (dihedral)\\
\hline
$G_4$ & $3$ & $4$ & $4,6$\\
\hline
$G_5$ & $3$& $8$ & $6,12$\\ 
\hline
$G_8$ & $4$& $6$ & $8,12$\\
\hline
$G_{12}$ & $2$& $12$ & $6,8$\\
\hline
$G_{13}$ & $2$& $18$ & $8,12$\\
\hline
$G_{16}$ & $5$& $12$ & $20,30$\\
\hline
$G_{20}$ & $3$& $20$ & $12,30$\\
\hline
$G_{22}$ & $2$& $30$ & $12,20$\\
\hline 
\end{tabular}

\end{center}
\begin{center}
Table 1
\end{center}

The next result generalises Proposition \ref{prop:factor}. Recall that for any finite group $\mu_N$ of roots
of unity in $\C$, $\gamma_j$ is the character of $\mu_N$ whose value at $\xi\in\mu_N$ is $\xi^j$.

\begin{prop}\label{prop:ind}
Suppose that $G$ is a uniform unitary reflection group in $V$, and that $Q,Q_0,d,e$ and $m$ are as in the previous paragraph.
Let $\Gamma=G\times\mu_m$, $F=\{v\in V\mid Q(V)=1\}$ and for $\zeta\in\mu_e$ define $F_\zeta$ by 
$F_\zeta:=\{v\in V\mid Q_0(v)=\zeta\}$. Then
\begin{enumerate}
\item $F=\amalg_{\zeta\in\mu_e}F_\zeta$.
\item The group $\Gamma$ permutes the $F_\zeta$ ($\zeta\in\mu_e$) transitively, and writing $F_0$ for $F_1$, 
the stabiliser of $F_0$ in $\Gamma$ is $\Gamma_0:=\ker(\ve\ot \gamma_d)$, where $\ve$ is the character of $G$
defined by $\ve(g)=\det_V(g)\inv\in\mu_e$.
\item Let $\Phi=\sum_{i=0}^{e-1}\ve^i\ot\gamma_{di}\in R_+(\Gamma)$. Then for each $j$, as elements of $R_+(\Gamma)$, we have
$$
H^j(F,\C)=\Ind_{\Gamma_0}^\Gamma(H^j(F_0,\C))=\Phi .H^j_0(F_0,\C),
$$
where $H^j_0(F_0,\C)$ is any element of $R_+(\Gamma)$ which restricts to $H^j(F_0,\C)\in R_+(\Gamma_0)$.
\item We have 
$$
P^\Gamma(F,t)=\Ind_{\Gamma_0}^\Gamma(P^{\Gamma_0}(F_0,t)=\Phi. P^\Gamma_0(F_0,t),
$$
where $P^\Gamma_0(F_0,t)$ is any element of $R_+(\Gamma)$ which restricts to $P^{\Gamma_0}(F_0,t)\in R_+(\Gamma_0)$.
\end{enumerate}
\end{prop}

\begin{proof} It follows from \cite[Cor. 9.17]{LT} that for $g\in G$, we have $g.Q_0=\ve(g) Q_0$.
The remainder of the proof now runs along similar lines to that of Proposition \ref{prop:factor}, and is therefore omitted.
\end{proof}

\begin{cor}\label{cor:h0} Let $G$ be any unitary reflection group acting in $V=\C^\ell$. For each hyperplane $H\in\CA_G$,
let $e_H$ be the order of its (cyclic) stabiliser $G_H$. Let $e$ be the greatest common divisor of $\{e_H\mid H\in\CA_G\}$,
 $m=\sum_{H\in\CA_G}e_H=\deg(Q=\prod_{H\in\CA_G}\ell_H^{e_H})$ and $d=\frac{m}{e}$. As usual we write $\Gamma=G\times\mu_m$.
Then as an element of $R(\Gamma)$, we have $H^0(F,\C)=\sum_{i=0}^{e-1}\ve^i\ot\gamma_{di}$, where 
$\ve(g)=\det_V(g)\inv$ for $g\in G$.
\end{cor}
\begin{proof}
Let $Q_0=\prod_{H\in\CA_G}\ell_H^{\frac{e_H}{e}}$. Then the argument in Proposition \ref{prop:ind}
shows that for all $j$ we have, writing $F_0$ for the variety $Q_0=1$ and $\Gamma_0$ for its stabiliser in $\Gamma$, 
\be\label{eq:genind}
H^j(F,\C)=\Ind_{\Gamma_0}^\Gamma(H^j(F_0,\C))=(\sum_{i=0}^{e-1}\ve^i\ot\gamma_{di}).H_0^j(F_0,\C),
\ee
where $H_0^j(F_0,\C)$ is any element of $R(\Gamma)$ which restricts to $H^j(F_0,\C)$.
But since the greatest common divisor of the integers $\frac{e_H}{e}$ is $1$, the variety $F_0$ is connected,
whence $H^0(F_0, \C)=1_{\Gamma_0}$.  The statement is now immediate. 
\end{proof}

In view of Proposition \ref{prop:ind}, in order to determine $P^\Gamma(F,t)$,
it suffices to compute $H^j(F_0,\C)\in R(\Gamma_0)$, and since, if $G$ acts in dimension $2$, $H^j(F_0,\C)\neq 0$
only for $j=0,1$ and $H^0(F_0,\C)=1\ot \gamma_0$, we need only compute $H^1(F_0,\C)$. For this we shall
apply the results in \cite{OS1,OS2}. 

\subsection{Determination of $H^1(F_0,\C)$}

We begin with the following consequences of the results of Orlik and Solomon  \cite{OS1,OS2}.

\begin{prop}\label{prop:os} Let $X,Y$ be coordinates in $V^*$; then $F_0=F_0(X,Y)$ is a homogeneous polynomial of 
degree $d$ with an isolated singularity at $(0,0)$.  
\begin{enumerate}
\item Let $I$ be the ideal of $\C[X,Y]$ generated by $\frac{\partial F_0}{\partial X}$ and $\frac{\partial F_0}{\partial Y}$.
Then there is an isomorphism $\theta:\C[X,Y]/I\lr H^1(F_0,\C)$ which satisfies the following condition.
For $x\in\GL(V)$ such that $x.Q_0=Q_0$, we have for any element $f\in \C[X,Y]/I$, $\theta(x.f)=\det_V(x)x.\theta(f)$.
\item With $x\in \GL(V)$ as in (1), we have
\be\label{eq:charos}
\Trace(x,\C[X,Y]/I)=(-1)^{k(x)}(d-1)^{k(x)}{\det}_V(x),
\ee
where $k(x)=\dim(\Fix_V(x))$.
\end{enumerate}
\end{prop}
\begin{proof}
The statement (1) follows from \cite[Theorem, p. 257]{OS1}, while (2) follows from \cite[Theorem (5.4)]{OS2}.
\end{proof}

The formula below is immediate from (2).
\begin{cor}\label{cor:trace}
The character $\chi$ of $\Gamma_0$ on $H^1(F_0,\C)$ is given by
$$
\chi\left((g,\zeta)\right)\Trace\left((g,\zeta),H^1(F_0,\C)\right)=(1-d)^{k(g,\zeta)},
$$
where $k(g,\zeta)=\dim(\Fix(g,\zeta))$.
\end{cor}

This permits the immediate determination of the monodromy action.

\begin{prop}\label{prop:mono}
The structure of $H^1(F_0,\C)$ as $\mu_d$-module, where $\mu_d(=1\times\mu_d)$ is the monodromy subgroup
of $\Gamma_0$, is given by
$$
H^1(F_0,\C)=(d-1)\gamma_0+(d-2)\sum_{i=1}^{d-1}\gamma_i.
$$
\end{prop}

\begin{proof}
Note that for $\xi\in\mu_d$, $k(1,\xi)=
\begin{cases}
0\text{ if }\xi\neq 1\\
2\text{ if }\xi= 1.\\
\end{cases}$

 It follows from Corollary \ref{cor:trace} that 
$$
\begin{aligned}
(\chi,\gamma_i)_{\mu_d}=&\frac{1}{d}\sum_ {\xi\in\mu_d}\chi((1,\xi))\xi^{-i}\\
=&\frac{1}{d}\left(d-1)^2+\sum_ {\xi\in\mu_d,\;\xi\neq 1}\chi((1,\xi))\xi^{-i}\right)\\
=&\frac{1}{d}\left(d-1)^2+\sum_ {\xi\in\mu_d,\;\xi\neq 1}\xi^{-i}\right)\\
=&\frac{1}{d}\left(d-1)^2+
\begin{cases}
-1\text{ if }i\neq 0\\
d-1\text{ if }i=0\\
\end{cases}\right)\\
=&
\begin{cases}
d-2\text{ if }i\neq 0\\
d-1\text{ if }i=0.\\
\end{cases}\\
\end{aligned}
$$
The result is now immediate.
\end{proof}

In Example \ref{ex:PDrk2} below we shall give the Poincar\'e-Deligne polynomial for the monodromy action in general.

\subsection{Determination of $H^1(F_0,\C)$ for the dihedral groups}\label{ss:dihedral} In this subsection we
give an explicit formula for the first cohomology of $F_0$ as an element of $R(\Gamma_0)$. For convenience, denote
$G(p,p,2)$ by $G(p)$. When $p=3,4$ and $6$, $G(p)$ is the Weyl group of type $A_2,B_2$ and $G_2$ respectively.
For other values $p>3$ $G(p)$ is a unitary (complex) reflection group; in each case, the reflection representation
of $G(p)$ may be realised as the set of matrices 
$$
\big\{
\begin{pmatrix}
\zeta & 0\\
0 & \zeta\inv\\
\end{pmatrix},
\begin{pmatrix}
0 & \zeta\\
\zeta\inv & 0\\
\end{pmatrix},
\zeta\in\mu_p
\big\}.
$$
The description of the irreducible characters of $G(p)$ depends on the parity of $p$.
If $p$ is odd, the character values are given in Table 2 below.

\begin{center}\label{table 2}
\begin{tabular}{|l||c|c|}
\hline
{Character}& {$\begin{pmatrix}
\zeta & 0\\
0 & \zeta\inv\\
\end{pmatrix}$}&$\begin{pmatrix}
0 & \zeta\\
\zeta\inv & 0\\
\end{pmatrix}$ \\
\hline 
\hline
$1$ & $1$ &$1$ \\
\hline
$\ve$ & $1$ & $-1$ \\
\hline
$\chi_k\;\;(1\leq k\leq \frac{p-1}{2})$ & $\zeta^k+\zeta^{-k}$& $0$ \\ 
\hline

\end{tabular}

\end{center}
\begin{center}
Table 2
\end{center}

If $p$ is even, the values of the irreducible characters of $G(p)$ are shown in Table 3.
\begin{center}\label{table 3}
\begin{tabular}{|l||c|c|}
\hline
{Character}& {$\begin{pmatrix}
\zeta & 0\\
0 & \zeta\inv\\
\end{pmatrix}$}&$\begin{pmatrix}
0 & \zeta\\
\zeta\inv & 0\\
\end{pmatrix}$ \\
\hline 
\hline
$1$ & $1$ &$1$ \\
\hline
$\ve$ & $1$ & $-1$ \\
\hline
$\delta$ & $\zeta^{\frac{p}{2}}$ & $\zeta^{\frac{p}{2}}$ \\
\hline
$\ve\delta$ & $\zeta^{\frac{p}{2}}$& $-\zeta^{\frac{p}{2}}$\\
\hline
$\chi_k\;\;(1\leq k\leq \frac{p-2}{2})$ & $\zeta^k+\zeta^{-k}$& $0$ \\ 
\hline

\end{tabular}

\end{center}
\begin{center}
Table 3
\end{center}

Now the irreducible characters of $\Gamma$ are all of the form $\rho\ot\gamma_i$, where $\rho$ is an irreducible character of $G(p)$
and $\gamma_i$ is the character $\zeta\mapsto\zeta^i$ of $\mu_{2p}$. Now $\Gamma_0$ is the kernel of the character $\ve\ot\gamma_p$.
It follows (see Proposition \ref{prop:ind} or \S\ref{ssec:factor}) that the irreducible characters of $\Gamma_0$ are the (distinct)
restrictions to $\Gamma_0$ of $\{\rho\ot\gamma_i\mid 0\leq i\leq p-1\}$, where $\rho$ runs over the irreducible characters of $G(p)$.
Using this parametrisation of the irreducible characters of $\Gamma_0$, we have the following result.

\begin{prop}\label{prop:dihed}
Let $G(p)$ ($p\geq 3$) be the reflection group $G(p,p,2)$ as described above, and let $\Gamma,\Gamma_0,F,F_0$ 
be as in Proposition \ref{prop:ind}. Then $H^1(F_0,\C)$ is given as an element of $R(\Gamma_0$ as follows.
\begin{enumerate}
\item If $p$ is odd, then
\be\label{eq:odd}
H^1(F_0,\C)=\sum_{\overset{0\leq i\leq p-1}{i\text{ odd }}}1\ot\gamma_i + \sum_{\overset{0<i\leq p-1}{i\text{ even }}}\ve\ot\gamma_i 
+\sum_{k=1}^{\frac{p-1}{2}}\sum_{\overset{0\leq i\leq p-1}{i\neq k,p-k}}\chi_k\ot\gamma_i.
\ee
\item If $p$ is even, then
\be\label{eq:even}
\begin{aligned}
H^1(F_0,\C)&=\ve\ot\gamma_0+2\sum_{\overset{0< i\leq p-1}{i\text{ even }}}\ve\ot\gamma_i + 
\sum_{\overset{0\leq i\leq p-1}{i+\frac{p}{2}\in 2\Z,\;\;i\neq\frac{p}{2}}}(\ve\delta+\delta)\ot\gamma_i \\
+&2\sum_{k=1}^{\frac{p-2}{2}}\sum_{\overset{0\leq i\leq p-1}{i+k\in 2\Z,\;\;i\neq k,p-k}}\chi_k\ot\gamma_i
+\sum_{k=1}^{\frac{p-2}{2}}\chi_k\ot(\gamma_k+\gamma_{p-k}).\\
\end{aligned}
\ee
\end{enumerate}   
\end{prop}
\begin{proof}
This is proved by simply carrying out the tedious task of computing the inner product
$(\chi,\rho\ot\gamma_i)_{\Gamma_0}$ of each irreducible character $\rho\ot\gamma_i$ 
of $\Gamma_0$ with the character $\chi$ of the representation of $\Gamma_0$ on $H^1(F_0,\C)$, as described in
Corollary \ref{cor:trace} above.
\end{proof}

This last result may be used to compute the weight polynomial $W^{\Gamma_0}(F_0,t)$ (see Proposition \ref{prop:wtdih}).

\subsection{Non-uniform groups of rank two}\label{ssec:nonuni} The values of $e_H=|G_H|$ for $H\in\CA_G$
for the non-uniform groups $G$ of rank two are given in Table 4 below.

\begin{center}\label{table 4}
\begin{tabular}{|l|c|c|c|}
\hline
{Group}& { values of $e_H$} & degrees of  $G$\\
\hline 
\hline
$G(pe,p,2)\;(p\geq 2,e\geq 3)$  &$2,e$ & $2e,pe$ (imprimitive)\\
\hline
$G_6$ & $2,3$ & $4,12$\\
\hline
$G_7$ & $2,3$ & $12,12$\\
\hline
$G_9$ & $2,4$ & $8,24$\\
\hline
$G_{10}$ & $3,4$ & $12,24$\\
\hline
$G_{11}$ & $2,3,4$ & $24,24$\\
\hline
$G_{14}$ & $2,3$ & $6,24$\\
\hline
$G_{15}$ & $2,3,4$ & $12,24$\\
\hline
$G_{17}$ & $2,5$ & $20,60$\\
\hline
$G_{18}$ & $3,5$ & $30,60$\\
\hline
$G_{19}$ & $2,3,5$ & $60,60$\\
\hline
$G_{21}$ & $2,3$ & $12,60$\\
\hline

\end{tabular}

\end{center}
\begin{center}
Table 4
\end{center}

Now we know the Euler characteristic $H^0(F,\C)-H^1(F,\C)$ by Corollary \ref{cor:euler}, and also $H^0(F,\C)$ by
Corollary \ref{cor:h0}. Thus we can deduce $H^1(F,\C)$. 

\section{Relations between monodromy and cohomology degree.}

In this section we prove some results which imply restrictions on the cohomology degree in
which monodromy of a specified order may appear.

\subsection{Two general relations} The following result, which goes back to Milnor \cite[pp.76-78]{Mi} and Oka \cite{Oka} may
 be found in \cite[Proposition 3.1.21]{D0}. We provide a short proof below, which shows that it is a simple
consequence of Proposition \ref{prop:z}.

\begin{prop}\label{prop:faithful}
Let $Q_0\in\C[x_1,\dots,x_n]$ be homogeneous of degree $d$. 
Let $F_0=\{x=(x_1,\dots,x_n)\in\C^n\mid Q_0(x_1,\dots,x_n)=1\}$ and $U=\{x \in \PP^{n-1}      
  \mid Q_0(x_1,\dots,x_n)\ne 0\}$ . Let $h:F_0\to F_0$ be given by $h(x)=\zeta x$, where $\zeta\in\C^*$ has order $d$.
Define the zeta function 
$$Z(t)=\prod_{p=0}^{n-1}\det\left(\Id-th^*|H^p(F_0,\C)\right)^{(-1)^p}.$$
Then 
$$Z(t)=(1-t^d)^{\frac{\chi(F_0)}{d}}= (1-t^d)^{\chi(U)},$$
where $\chi$ denotes Euler characteristic.
\end{prop}
\begin{proof}
Let $H=\langle h\rangle$. 
For any element $y\in H$, we have the evaluation homomorphism $\ve_y:R(H)\to\C$,
where $R(H)$ is the Grothendieck ring of $H$; this is defined by $\ve_y(\theta)=\Tr(y,\theta)$, which we
generally write simply as $\theta(y)$.
This homomorphism extends in the obvious way to $R(H)[[t]]$, and
we consider $Z(t)$ as the evaluation at $h\in H$ of the element $Z^H(t):=
\prod_{p=0}^{n-1}\det\left(\Id-tH^p(F_0,\C)\right)^{(-1)^p}$ of $R(H)[[t]]$.

Taking the formal logarithm of $Z^H(t)$, using the identity $\ln(1-x)=-(x+\frac{x^2}{2}+\frac{x^3}{3}+\frac{x^4}{4}+\dots)$,
we see that 
$$
-\ln(Z^H(t))(h)=\sum_{j=1}^\infty\left( \chi^H(h^{j}) \right)\frac{t^j}{j}.
$$
But by Proposition \ref{prop:z}, the equivariant Euler characteristic $\chi^H=\chi(U)\Reg_H$, whence it follows that
the above expression simplifies to
$$
-\ln(Z^H(t))(h)=\chi(U)\sum_{i=1}^\infty\frac{t^{id}}{i},
$$
and this last expression is evidently equal to $-\ln\left( (1-t^d)^{\chi(U)} \right)$.
\end{proof}

We will apply this result in the case when $G$ is a reflection group such that $e_H=2$ for each reflecting hyperplane
 $H$ and $Q_0=\prod_{H\in\CA}\ell_H$ as above. Define $Y=\cup_{H\in\CA}H$ and let $N =\PP(Y)$ be the corresponding
 hypersurface in the projective space $\PP(V)$. For any character $\gamma \in \widehat \mu_d$, with  $d=\frac{m}{2}$, it is known that
$$\dim H^k(F_0)_\gamma=\dim H^k(U, L_\gamma)$$
where $L_\gamma$ is the rank one local system on $U$ with monodromy around each hyperplane in $N$ given by multiplication
 by $\gamma(\zeta_d)$, see for instance  \cite[Prop. 6.4.6]{D2}.

\medskip

In order to prove vanishing results for the twisted cohomology groups $ H^k(U, L_\gamma)$ we need a good compactification 
of $U$, which is the same as an embedded resolution of $N$.
To describe this, we need the following basic notion about hyperplane arrangements.
An {\it edge}  $X$ is an intersection of hyperplanes in $\A$.  An edge $X$ is {\em
dense}  if the subarrangement of hyperplanes $\A_X$ containing it is
irreducible, i.e. the hyperplanes in $\A_X$  cannot be partitioned into two nonempty sets
so that after a linear change of coordinates on $V$ hyperplanes in different  sets
are are given by linear equations in different, disjoint sets of coordinates.  In particular, any hyperplane is a dense edge.

The condition of $X$ being dense is a combinatorially determined
condition which can be checked in a neighborhood of a given edge, see
\cite{STV}.   Let $D(\A)$ denote the set of dense edges
of the arrangement $\A$.

\medskip

There is a canonical way to obtain an embedded resolution of the
divisor $N$ in  $\PP(V)$.  First, blow up the points on $\PP(V)$  which correspond to the dense 1-dimensional edges
of $\A$ to obtain a map $p_1:Z_1 \to \PP(V)$.  Then, blow up all the
proper transforms under $p_1$ of projective lines  on $\PP(V)$ corresponding to
the dense 2-dimensional edges in $D(\A)$.  Continuing in this way, we
get a map $p=p_{\ell-2}:Z=Z_{\ell-2} \to  \PP(V)$ which is an embedded
resolution of the divisor $N$ in  $\PP(V)$, if $\ell=\dim V$.    Then,
$D=p^{-1 }(N)$ is a normal crossing divisor in $Z$, with smooth
irreducible components $D_X$ corresponding to the edges $X \in
D(\A)$ with $\dim X>0$.  Furthermore, the map $p$ induces an isomorphism 
$Z\setminus D =U$, see \cite{OT2,STV} for details.

Let $L$ be a rank one local system  on $Z\setminus D =U$. Then the monodromy of $L$ about an
 irreducible component $D_X$ of the normal crossing divisor $D$
is the nonzero complex number $c_X$ obtained by taking the products of the complex numbers $c_H$ which represent the 
monodromies of $L$ about the hyperplanes $H \in \A$ such that $X \subset H$.

In particular, for $L=L_\gamma$ as above, we get  $c_X=\gamma(\zeta_d)^{m_X}$, where $m_X$ denotes the number 
of hyperplanes $H \in \A$ such that $X \subset H$.

The following vanishing result follows from  \cite[Theorem 6.4.18 and Remark 6.4.20]{D2}.
See also \cite{Li5} and \cite{CDO}.

\begin{thm} \label{thm:nonres}
Assume that for a given rank one local system $L$, there is a hyperplane $H \in \A$ such that $c_X\ne 1$ for any dense
edge $X \in D(\A)$ with $X \subseteq H$ and $\codim X \leq c$.  Then $H^p(U,L)=0$ for any $p$ with $0 \leq p < c$. 
\end{thm}

Taking into account standard affine Lefschetz theorems, this implies the next result which relates cohomology degree and monodromy order. 

\begin{prop}\label{prop:mondeg} Let $\CA$ be an arrangement of $d$ hyperplanes in $\C^\ell$ and let $Q_0=\prod_{H\in\CA}\alpha_H$
where $\alpha_H$ is a linear form corresponding to $H\in\CA$.
Let $F_0$ be the corresponding Milnor fibre, and
let $\gamma$ be a character of $\mu_d$ which occurs with non-zero
multiplicity in $H^{s}(F_0,\C)$ for some $s >0$. Then there is a dense edge $X$ in $D(\A)$ such that $\codim X \leq s+1$ and 
the order of $\gamma$ divides the multiplicity $m_X$.
\end{prop}

Recall that if $G$ is a reflection group in $V$, then writing  $E=\cap_{H \in \A}H$,  the rank of the arrangement $\A$
is $r=\codim E$. The corresponding Milnor fibre $F_0$ is naturally isomorphic to a product $F_1 \times E$, 
with $F_1$ a hypersurface in an $r$-dimensional affine space. In particular $H^k(F_0)=H^k(F_1)=0$ for $k \geq r$.

Moreover, for any $X \in D(\A)$ with $\dim X>\dim E$ one has $m_X<d$. This yields the following statement.
\begin{cor}\label{cor:faithfulgen}
In the notation above, $H^j(F_0,\C)=0$ if $j\geq r$. Moreover
if $\gamma$ is a faithful character of $\mu_d$, i.e., $\gamma$ has order $d$, then 
$\gamma$ occurs in $H^{j}(F_0)$ with non-zero multiplicity if and only if $j=r-1$.
The multiplicity of $\gamma$ in $H^{r-1}(F_0)$ is equal to $|\chi(F_0)/d|=|\chi(U)|$.
\end{cor}

This yields the following result.
\begin{cor}\label{cor:faithpart}
Let $G$ be a reflection group in $V=\C^\ell$, and assume $e_H=2$ for every hyperplane $H$ of $G$.
Let $r$ be the rank of the corresponding hyperplane arrangement, and 
let $\gamma$ be a faithful character of $\mu_{\frac{m}{2}}$
where $m=\sum_{H\in\CA_G}e_H$. 
Then $\gamma$ occurs in $H^j(F_0,\C)$ with non-zero multiplicity only if $j=r-1$, and the multiplicity of
$\gamma$ in $H^{r-1}(F_0,\C)$ is $|e(G)|=\prod_{i\geq 2}(m_i^*-1)$.
\end{cor}
\begin{proof}
Apply Corollary \ref{cor:faithfulgen} to the polynomial $Q_0$ of \S\ref{ssec:factor} with $d=\frac{m}{2}$.
The first assertion is immediate, and in this case it follows from Proposition \ref{prop:ffree} that
$\chi(F_0)/d=\chi(F)/m=\chi(F^{\mu_m})=\chi(U)=(-1)^{r-1}(m_2^*-1)(m_3^*-1)\dots(m_{r-1}^*-1)$.
\end{proof}

\begin{rk}
\begin{enumerate}
\item Note that Corollaries \ref{cor:faithfulgen} and \ref{cor:faithpart} are consistent with Proposition
\ref{prop:ffree}(2).
\item It may be easily shown, using \cite[Cor. 10.4.1]{LT} that 
$$
e(G)=(-1)^\ell\sum_{g\in G}\det{_V} g\inv \dim(\Fix g).
$$
\end{enumerate}
\end{rk}

\begin{ex}\label{ex:dense}  We take $G=Sym_{\ell +1}$ acting by permutations on  $\C^{\ell+1}$.
In this case the rank of the corresponding braid arrangement $\A_{\ell}$ is $r=\ell$.

The edges $X \in L(\A)$ are given, up to the induced action of $G$, by partitions $\kappa=(k_1\geq k_2\geq\dots k_s\geq 1)$ of $\ell+1$. 
The dense edges correspond to partitions satisfying $k_2=...=k_s=1$. If $X$ is associated to such a partition, then $X$ has codimension 
$\ell+1-s $ and multiplicity $m_X=\binom{\ell+2-s}{2}$.

\end{ex}

This proves the following.

\begin{cor}\label{cor:lowdeg}
Let $G=Sym_{\ell +1}$ acting by permutations on  $\C^{\ell+1}$ , and as above, write $Q_0=\prod_{H\in\CA_G}\ell_H$.
Let $F_0$ be the corresponding Milnor fibre given by $Q_0(x)=1$. The group $\mu_d$ acts 
on $H^*(F_0,\C)$ where $d=\frac{m}{2}=\binom{\ell+1}{2}$ ; suppose $\gamma\in\wh\mu_d$ appears with non-zero multiplicity in 
$H^i(F_0,\C)$. Then
\begin{enumerate}
\item If $i=0$ then $\gamma=\gamma_0=1_{\mu_d}$.
\item If $i=1$ then $|\gamma|$ divides $g.c.d.(3,d)$.
\item If $i=2$ then $|\gamma|$ divides $g.c.d.(6,d)$.
\item If $i=3$ then $|\gamma|$ divides either $g.c.d.(6,d)$ or  $g.c.d.(10,d)$.
\item If $i=4$ then $|\gamma|$ divides at least one of the following integers: $g.c.d.(6,d)$, $g.c.d.(10,d)$ or $g.c.d.(15,d)$.

\end{enumerate}
\end{cor}

\begin{proof}
We apply Proposition \ref{prop:mondeg} and Example \ref{ex:dense}.
\end{proof}

\begin{ex}\label{ex:i=2} 
To compare the above result with the computations in \cite{Se}, we see that for $i=2$, eigenvalues of order $6$ occur if $\ell=3$ (then $d=6$), 
eigenvalues of order $2$ occur for $\ell=4$ (then $d=10$), and eigenvalues of order $3$ occur for $\ell=5, 6$ (then $d= 15, 21$),
i.e. in these cases our results are sharp.

However, it is known that for $\ell>6$ only the trivial character $\gamma_0$ occurs in
$H^2(F_0,\C)$, \cite{Se1}, and this fact cannot be obtained with our techniques so far.

\end{ex}

\medskip

To apply the above results, the following lemma will be useful. In the statement, note that 
$H^i(U,\C)$ is regarded as a $G$-module, and its structure is well known (cf. \cite{L87,L1,L96}).

\begin{lem}\label{lem:trivmon}

Let $Q$ etc. be as in \S\ref{sec:intro} and \S\ref{ssec:factor}. For elements $A,B\in R(\Gamma)[t]$
say that $A\geq B$ if $A-B\in R_+(\Gamma)[t]$.

\begin{enumerate}
\item We have 
$P^\Gamma(F,t)\geq (1\ot\gamma_0 + \ve\ot\gamma_{\frac{m}{2}})\left(P^G(U,t)\ot\gamma_0\right)$.

\item If only the trivial monodromy
character $\gamma_0\in\wh{\mu_{\frac{m}{2}}}$ occurs in $H^i(F_0,\C)$, 
then as $\Gamma$-module, we have
$$
H^i(F,\C)\simeq (1\ot\gamma_0+ \ve\ot\gamma_{\frac{m}{2}})(H^i(U,\C)\ot\gamma_0).
$$
In this case, as $\Gamma_0$-module, $H^i(F_0,\C)\simeq H^i(U,\C)\ot\gamma_0$.
\item The multiplicity of $1_{\mu_{\frac{m}{2}}}$ in $H^i(F_0,\C)$ is $b=\dim(H^i(U,\C))$.
\end{enumerate}
\end{lem}
\begin{proof} We have $F/\mu_m\simeq F_0/\mu_{\frac{m}{2}}\simeq U$, whence
$H^i(U,\C)$ is isomorphic to the $\gamma_0$-isotypic component of $H^i(F_0,\C)$
for each $i$. 
Further, using the fact that 
$H^i(F,\C)=\Ind_{\Gamma_0}^\Gamma(H^i(F_0,\C))=\Ind_{\Gamma_0}^\Gamma(\rho\ot\gamma_0)$ (see Proposition
\ref{prop:factor}) is easily shown
that $H^i(F,\C)\geq (1\ot\gamma_0 + \ve\ot\gamma_{\frac{m}{2}})(\rho\ot\gamma_0)$,
where $\rho$ is the $\gamma_0$-isotypic part of $H^i(F_0,\C)$.
 Hence $\rho\simeq H^i(U,\C)$, and the first statement follows. 
 
Parts (2) and (3) follow easily.
\end{proof}

Note that as we have seen (Proposition \ref{prop:ffree}), the Poincar\'e polynomial $$P^G(U,t;1)=\sum_i\dim H^i(U,\C)t^i$$ of $U$ is known 
to be equal to $\prod_{i\geq 2}(1+m_i^*t)$, where the $m_i^*$ are the coexponents of $G$. Thus 
the multiplicities in part (3) of the Lemma are known.

\subsection{Application to the symmetric groups} We again take $G=\Sym_{\ell+1}$, with $Q,Q_0$ etc.
as in \S\ref{ssec:factor}.

\begin{ex}\label{ex:sym4}
We illustrate the use of the above results by computing $P^\Gamma(F,t)$ in the
case where $G=\Sym_4$.

Denote the irreducible representations of $\Sym_4$ by $1,\ve,\rho,\ve\rho$ and $\sigma$,
where $1,\ve$ are the trivial and alternating representations respectively, $\rho,\ve\rho$
are the reflection representation and its tensor with $\ve$ respectively, and $\sigma$ is the 
$2$-dimensional irreducible representation corresponding to the partition $(2,2)$.

Our notation for the characters of the cyclic group $\mu_m\subset\C^\times$ is as above:
$\gamma_1$ denotes the generating character, which is just the inclusion $:\mu_m\hookrightarrow\C^\times$
and for any $i\in\Z$, $\gamma_i=\gamma_1^i$. Notice that this notation is independent of $m$,
so that restriction of characters from $\mu_m$ to $\mu_d$ (for $d|m$) is trivial. The 
distinct irreducible characters of $\mu_m$ are $\gamma_0=1_{\mu_m},\gamma_1,\dots,\gamma_{m-1}$.

Every irreducible representation of $\Gamma=G\times\mu_m$ is of the form $\alpha\ot\gamma_i$,
where $\alpha\in I(G)$.

First we apply \eqref{eq:thm-sym} with $\ell=3$
to obtain the following expression for $\chi^\Gamma(F)\in R(\Gamma)$. The factorisation is
in accord with Proposition \ref{prop:factor}.

\be\label{eq:eulers4}
\begin{aligned}
\chi^\Gamma(F)=&1\ot(\gamma_0+\gamma_3+\gamma_9)+\ve\ot(\gamma_3+\gamma_6+\gamma_9)
+\rho\ot(\gamma_2+\gamma_6+\gamma_{10})+\ve\rho\ot(\gamma_0+\gamma_4+\gamma_8)\\
&+\sigma\ot(\gamma_1+\gamma_5+\gamma_7+\gamma_{11})
-\left[1\ot(\gamma_2+\gamma_{10}) + \ve\ot(\gamma_4+\gamma_8)+\sigma\ot(\gamma_0+\gamma_6)\right]\\
=&(1\ot\gamma_0+\ve\ot\gamma_6)\{1\ot(\gamma_0+\gamma_3+\gamma_9)
+\rho\ot(\gamma_2+\gamma_6+\gamma_{10})+\sigma\ot(\gamma_1+\gamma_{11})\\
&-[1\ot(\gamma_2+\gamma_{10})+\sigma\ot\gamma_0]\}.\\
\end{aligned}
\ee

We know that as $\Gamma$-module, $H^0(F,\C)=\Ind_{\Gamma_0}^\Gamma(1)=1\ot\gamma_0+\ve\ot\gamma_6$.
Further, since $\chi^\Gamma(F)=H^0(F,\C)-H^1(F,\C)+H^2(F,\C)$
it follows from \eqref{eq:eulers4} that as $\Gamma$-module, $H^1(F,\C)$ contains the module
$$
M_1:=1\ot(\gamma_2+\gamma_{10}) + \ve\ot(\gamma_4+\gamma_8)+\sigma\ot(\gamma_0+\gamma_6),
$$
and that $H^2(F,\C)$ contains the module 
$$
\begin{aligned}
&M_2:=1\ot(\gamma_3+\gamma_9)+\ve\ot(\gamma_3+\gamma_9)
+\rho\ot(\gamma_2+\gamma_6+\gamma_{10})+\ve\ot\rho(\gamma_0+\gamma_4+\gamma_8)\\
&+\sigma\ot(\gamma_1+\gamma_5+\gamma_7+\gamma_{11}).\\
\end{aligned}
$$

Next we note that the $G$-module structure of $H^*(U,\C)=H^*(F/G,\C)$ is explicitly 
given for all the symmetric groups in \cite{L87}. In our case we have 
\be\label{eq:hus4}
P^G(U,t)=1+(\rho+\sigma)t+(\rho+\ve\rho)t^2.
\ee

Hence by Lemma \ref{lem:trivmon}(1), $H^1(F,\C)$ also contains the module
$(1\ot\gamma_0+\ve\ot\gamma_6)(\rho\ot\gamma_0)$, which has no common irreducible
constituent with $M_1$, whence $H^1(F,\C)$ contains 
$M_1+(1\ot\gamma_0+\ve\ot\gamma_6)(\rho\ot\gamma_0)$.

But inspection of \cite[Table 2]{Se} shows that $\dim H^1(F_0,\C)=7$, whence $\dim H^1(F,\C)=14$,
and by counting dimensions it follows that 
$$
H^1(F,\C)=M_1+(1\ot\gamma_0+\ve\ot\gamma_6)(\rho\ot\gamma_0).
$$

Similarly, applying Lemma \ref{lem:trivmon} to $H^2(F,\C)$, it follows that
$(1\ot\gamma_0+\ve\ot\gamma_6)(\rho\ot\gamma_0)$ is also a submodule of
$H^2(F,\C)$, and has no common constituent with $M_2$. Using an argument
similar to that above, we see that
$$
H^2(F,\C)=M_2+(1\ot\gamma_0+\ve\ot\gamma_6)(\rho\ot\gamma_0).
$$

We have therefore proved the following
\begin{prop}\label{prop:s4}
The $\Gamma$-module structure of $H^*(F,\C)$ is given by the following equations in $R(\Gamma)$.
\be\label{eq:s4h0}
H^0(F,\C)\simeq 1\ot\gamma_0+\ve\ot\gamma_6;
\ee
\be\label{eq:s4h1}
\begin{aligned}
H^1(F,\C)\simeq 1\ot(\gamma_2+\gamma_{10}) + &\ve\ot(\gamma_4+\gamma_8)+\sigma\ot(\gamma_0+\gamma_6)\\
+&\rho\ot\gamma_0+\ve\rho\ot\gamma_6;\\
\end{aligned}
\ee
\be\label{eq:s4h2}
\begin{aligned}
H^2(F,\C)\simeq 1\ot(\gamma_3+\gamma_9)+&\ve\ot(\gamma_3+\gamma_9)
+\rho\ot(\gamma_0+\gamma_2+\gamma_6+\gamma_{10})\\
+&\ve\rho\ot(\gamma_0+\gamma_4+\gamma_6+\gamma_8)
+\sigma\ot(\gamma_1+\gamma_5+\gamma_7+\gamma_{11}).\\
\end{aligned}
\ee

In terms of Poincar\'e polynomials, this result may be restated as follows. We have
$P^\Gamma(F,t)=(1\ot\gamma_0+\ve\ot\gamma_6)P_0^\Gamma(F,t)$, where 
\be\label{eq:s4poin}
\begin{aligned}
P_0^\Gamma=1\ot\gamma_0+&[1\ot(\gamma_2+\gamma_{10}) +(\sigma+\rho)\ot\gamma_0]t\\
+&\{1\ot(\gamma_3+\gamma_9)
+\rho\ot(\gamma_0+\gamma_2+\gamma_6+\gamma_{10})
+\sigma\ot(\gamma_1+\gamma_{11})\}t^2.\\
\end{aligned}
\ee
\end{prop}
\end{ex}

Note that by Proposition \ref{prop:factor} $P^{\Gamma_0}(F_0,\C)$ is the restriction
to $\Gamma_0$ of $P_0^\Gamma$ of the above proposition, and that by the proof of 
Proposition \ref{prop:factor}, each irreducible representation of $\Gamma$ restricts to
an irreducible representation of $\Gamma_0$. The $\mu_{6}$-structure of $H^*(F_0,\C)$,
given in \cite{Se}, is an easy consequence. The formulae above may be seen to be consistent
with both \cite[Table 2]{Se} and Proposition \ref{prop:dps}.

\begin{ex}\label{ex:s5}
A similar, but slightly longer, calculation yields the following result for the symmetric group
$\Sym_5$, which we set out below without details, for comparison with some 
of the general results we describe later.

As usual, we label the irreducible representations of $\Sym_5$ with the partitions $\lambda$
of $5$. Thus $\chi^\lambda$ corresponds to $\lambda$, and in accord with the usual conventions,
$\chi^{(5)}=1$, the trivial representation, $\chi^{1^5}=\ve$, the alternating representation,
$\chi^{(4,1)}=\rho$, the reflection representation, and $\ve\ot\chi^{\lambda}=\chi^{\lambda'}$,
where $\lambda'$ is the partition dual to $\lambda$. The irreducible representations
$\chi^{(5)}=1,\chi^{(1^5)}=\ve,\chi^{(4,1)}=\rho,\chi^{(2,1^3)}=\ve\rho,
\chi^{(3,2)},\chi^{(2^2,1)}$ and $\chi^{(3,1^2)}$ thus have respective dimensions
$1,1,4,4,5,5$ and $6$. Notation for the irreducible characters of $\mu_{20}$ is
as above. The character $\gamma_i$ takes $\zeta\in\mu_{20}$ to $\zeta^i$.

\begin{prop}\label{prop:s5}
Let $\Gamma=\Sym_5\times\mu_{20}$, and let $F, F_0,\Gamma_0$ be as above for the hyperplane arrangement
of type $A_4$. The cohomology $H^i=H^i(F,\C)$ is given as a $\Gamma$-module by the formulae below.

\begin{enumerate}
\item $H^0=1\ot\gamma_0+\ve\ot\gamma_{10}$.
\item
$H^1=\rho\ot\gamma_0+\ve\rho\ot\gamma_{10}+\chi^{(3,2)}\ot\gamma_0+\chi^{(2^2,1)}\ot\gamma_{10}.$
\item
$H^2=1\ot(\gamma_5+\gamma_{15})+\ve\ot(\gamma_5+\gamma_{15})+\rho\ot\gamma_0+\ve\rho\ot\gamma_{10}
+\chi^{(3,2)}\ot(\gamma_0+\gamma_{10})+\chi^{(2^2,1)}\ot(\gamma_0+\gamma_{10})
+2\chi^{(3,1^2)}\ot(\gamma_0+\gamma_{10}).$
\item 
$H^3=1\ot(\gamma_2+\gamma_6+\gamma_{14}+\gamma_{18})+\ve\ot(\gamma_4+\gamma_8
+\gamma_{12}+\gamma_{16})+\rho\ot(\gamma_0+\gamma_5+\gamma_{10}+\gamma_{15})+\ve\rho\ot(\gamma_0+\gamma_5+\gamma_{10}+\gamma_{15})
+\chi^{(3,2)}\ot(\gamma_0+\gamma_{4}+\gamma_{8}+\gamma_{10}+\gamma_{12}+\gamma_{16})
+\chi^{(2^2,1)}\ot(\gamma_0+\gamma_{2}+\gamma_{6}+\gamma_{10}+\gamma_{14}+\gamma_{18})
+\chi^{(3,1^2)}\ot(\gamma_0+\gamma_{1}+\gamma_{3}+\gamma_{7}+\gamma_{9}+\gamma_{10}+\gamma_{11}
+\gamma_{13}+\gamma_{17}+\gamma_{19}).$
\item
In terms of Poincar\'e polynomials, this result may be stated as follows. We have
$P^\Gamma(F,t)=(1\ot\gamma_0+\ve\ot\gamma_{10})P_0^\Gamma(F,t)$, where 

$P_0^\Gamma(F,t)=1+[\rho\ot\gamma_0+\chi^{(3,2)}\ot\gamma_0]t+
[1\ot(\gamma_5+\gamma_{15})+\rho\ot\gamma_0+\chi^{(3,2)}\ot(\gamma_0+\gamma_{10})+
\chi^{(3,1^2)}\ot(\gamma_0+\gamma_{10})]t^2+[1\ot(\gamma_2+\gamma_6+\gamma_{14}+\gamma_{18})
+\rho\ot(\gamma_0+\gamma_5+\gamma_{10}+\gamma_{15})+
\chi^{(3,2)}\ot(\gamma_0+\gamma_{4}+\gamma_{8}+\gamma_{10}+\gamma_{12}+\gamma_{16})+
\chi^{(3,1^2)}\ot(\gamma_0+\gamma_{1}+\gamma_{3}+\gamma_{7}+\gamma_{9})]t^3.$
\end{enumerate}
\end{prop}
\end{ex}

\section{Relations with mixed Hodge theory.}\label{s:mh}

In this section we investigate the relationship between the $\Gamma$-action and the
mixed Hodge structure (henceforth denoted MHS) on the Milnor fibre cohomology. Refer to \cite{PS} 
for general notions and results concerning the MHS.

\subsection{Equivariant Hodge-Deligne polynomials and spectra}

To do this, we first recall the definition of the equivariant Hodge-Deligne polynomials $HD^{\Gamma}(X)$
associated to a finite group $\Gamma$ acting (algebraically) on a complex algebraic variety $X$, see \cite{DL2} . 

More precisely, let $X$ be a quasi-projective variety over $\C$ and consider the Deligne MHS on the rational cohomology groups
$H^*(X,\Q)$ of $X$. Since this MHS is functorial with respect to algebraic mappings, if $\Gamma$ acts
algebraically on $X$, each of the graded pieces
\begin{equation} 
\label{eq:hpq}
H^{p,q}(H^j(X,\C)):=Gr_F^pGr^W_{p+q}H^j(X,\C)
\end{equation}
becomes a $\Gamma$-module via the action induced by the action defined in \eqref{eq:cohoaction}, and these modules are the building 
blocks of the polynomial $HD^{\Gamma}(X;u,v)\in R(\Gamma)[u,v]$, defined by
\begin{equation} 
\label{eq:hd}
HD^{\Gamma}(X;u,v)=\sum_{p,q}E^{\Gamma;p,q}(X)u^pv^q,
\end{equation}
where $E^{\Gamma;p,q}(X)=\sum_j(-1)^jH^{p,q}(H^j(X,\C))\in R(\Gamma)$. Note that 
$$HD^{\Gamma}(X;1,1)=\chi^{\Gamma}(X).$$
One may consider an even finer (and hence harder to determine) invariant, namely the equivariant Poincar\'e-Deligne polynomial
\begin{equation} 
\label{eq:pd}
PD^{\Gamma}(X;u,v,t)=\sum_{p,q,j} H^{p,q}(H^j(X,\C))u^pv^qt^j \in R_+(\Gamma)[u,v,t].
\end{equation}
Clearly one has $PD^{\Gamma}(X;u,v,-1)=HD^{\Gamma}(X;u,v)$ and $PD^{\Gamma}(X;1,1,t)=P^{\Gamma}(X;t)$.

Coming back to our setting, Proposition \ref{prop:factor} above extends (with exactly the same proof) to yield the following result.

\begin{prop}\label{prop:factor2} Let $G$ be a reflection group such that $e_H=2$ for each reflecting hyperplane
$H$.
In the notation above, we have
$$
PD^\Gamma(F;u,v,t)=(1\ot\gamma_0 + \ve\ot\gamma_{\frac{m}{2}}) P_0^\Gamma(F;u,v,t),
$$
where $P_0^\Gamma(F;u,v,t)$ is 
a
 polynomial in the ring $R_+(\Gamma)[u,v,t]$ whose restriction to $\Gamma_0$
is $PD^{\Gamma_0}(F_0;u,v,t)$. 

In particular,
$$
HD^{\Gamma}(F;u,v))=(1\ot\gamma_0 + \ve\ot\gamma_{\frac{m}{2}})HD_0^{\Gamma}(F;u,v)
$$
where $HD_0^\Gamma(F;u,v)$ is 
any
 polynomial in the ring $R(\Gamma)[u,v]$ whose restriction to $\Gamma_0$
is $HD^{\Gamma_0}(F_0;u,v)$.
\end{prop}
It follows that the crucial question is how to compute (or at least get information about) the polynomials
 $PD^{\Gamma_0}(F_0;u,v,t)$ and $HD^{\Gamma_0}(F_0;u,v)$.
Consider the natural inclusion $\mu_d \to \Gamma_0$ given by $\zeta \to (1,\zeta) \in \Gamma_0 \subset G \times \mu_m$,
where as in \S 3, $d=\frac{m}{2}=\deg(Q_0)$. It induces a ring homomorphism 
\begin{equation} 
\label{eq:theta}
\theta: R(\Gamma_0) \to R(\mu_d),
\end{equation}
by restriction of representations, i.e. such that $\theta (W \ot \gamma_k)=(\dim W)\gamma_k$, for any $G$-module $W$
and character $ \gamma_k\in \widehat \mu_m$; note that on the right side of this relation, 
$ \gamma_k=(\gamma_1)^k\in \widehat \mu_d$.
This morphism $\theta$ can be used to relate known facts on the usual monodromy action 
(i.e. $\mu_d$-action) on the cohomology of the usual (connected) Milnor fibre $F_0$ and the 
$\Gamma_0$-action on the same cohomology. 

A simple example is that Proposition \ref{prop:faithful} may be reformulated as follows.
\begin{equation} 
\label{eq:theta1}
\theta (\chi^{\Gamma_0}(F_0))=\chi(U)\Reg_{\mu_d}.
\end{equation}
It is interesting to apply this restriction morphism to the equivariant Hodge-Deligne polynomial $HD^{\Gamma_0}(F_0;u,v)$. We get
\begin{equation} 
\label{eq:theta2}
\theta (HD^{\Gamma_0}(F_0;u,v))=\sum_{p,q} 
(\sum_{j,s}(-1)^j\dim H^{p,q}(H^j(F_0,\C))_{ \lambda(s,d)}^{h^j}\gamma_s)u^pv^q,
\end{equation}
with $ \lambda(s,d)=\gamma_s(\zeta_d)$ and $H^{p,q}(H^j(F_0,\C))_{ \lambda(s,d)}^{h^j}$ denoting the $ \lambda(s,d)$-eigenspace 
of the monodromy operator $h^j$ (which corresponds to the multiplication by $\zeta_d$ in the $\mu_d$-action defined by \eqref{eq:cohoaction} 
and the obvious inclusion $\mu_d \to \mu_m \to \G$).

Taking $v=1$ in this expression, we get an element in $R(\mu_d)[u]$, given by
\begin{equation} 
\label{eq:theta3}
\theta (HD^{\Gamma_0}(F_0;u,1))=\sum_{p} 
(\sum_{j,s}(-1)^j\dim Gr_F^{p}(H^j(F_0,\C))_{ \lambda(s,d)}^{h^j}\gamma_s)u^p=\sum_{p,s} c_{p,s}\gamma_su^p,
\end{equation}
for some coefficients $c_{p,s}\in\Z$.

Recall that $T^j=(h^j)^{-1}$, the inverse of the monodromy operator on the $j$-th cohomology of $F_0$,
is the monodromy of the local system on $\C^*$ coming from the constructible sheaf $R^jQ_{0*}\C_V$, see section 2 in \cite{DS} and 
compare to \eqref{eq:cohoaction}. Then note that 
\begin{equation} 
\label{eq:monoes}
Gr_F^{p}(H^j(F_0,\C))_{\lambda(s,d)}^{h^j}= Gr_F^{p}(H^j(F_0,\C))_{\beta(s,d)}^{T^j},
\end{equation}
where the exponents $h^j$ and $T^j$ show which linear map is used to compute eigenspaces,
$ \lambda(s,d)=\exp(2\pi i \al_0)$ and $\beta(s,d)= \exp(-2\pi i \al_0)$,
with $\al_0= \frac{s}{d}\in [0,1)$.

To put the above into context, we recall the definition of the spectrum $Sp(\A)$ of an essential hyperplane 
arrangement $\A:Q_0(x)=0$ in the $\ell$-dimensional vector space $V$:
\begin{equation} 
\label{eq:spec}
Sp(\A)=\sum_{\al \in \Q}m_{\al}t^{\al},
\end{equation}
with $m_{\al}=\sum_j(-1)^{j-\ell+1}\dim Gr_F^p\tilde H^j(F,\C)_{\beta}^{T^j}$ where $p=[\ell-\al]$ and $\beta=\exp(-2\pi i\al)$. 
Here $[y]$ denotes the integral part of a real number $y$, i.e. the largest integer $z$ such that $z \leq y$.
A recent result of Budur and Saito in \cite{BS} asserts that $Sp(\A)$
is determined by the combinatorics, i.e. by the lattice $L(\A)$.

For any vector space $W=\C^N$, linear transformation $\phi$ of $W$ and $\xi\in\C$ write $W[\phi,\xi]:=\{w\in W\mid \phi(w)=\xi w\}$
for the corresponding eigenspace. Writing $\mu_d=\langle h\rangle$, where $h$ is, as above, the monodromy induced by
multiplication on $F$ by $\zeta_d=\exp{\frac{2\pi i}{d}}$, we evidently have, for any $\mu_d$-module $M$,
$$
\dim M[h,\zeta_d^i]=(M,\gamma_i)_{\mu_d},
$$
where the right side denotes the multiplicity pairing in $R(\mu_d)$.
The same formula applies when $M$ is taken to be any virtual $\mu_d$-module, with $\dim$ suitably interpreted.
\begin{prop}\label{prop:spec-alt}
Let $M^{(p)}$ be the virtual $\Gamma_0$ module defined by
$$
M^{(p)}=(-1)^{\ell-1}\sum_j(-1)^j Gr_F^p \wt H^j(F,\C).
$$
In particular $M^{(p)}$ is (by restriction) a $\mu_d$-module, and we have the following equation in $\Z[t^{\frac{1}{d}}]$.
$$
Sp(\A)=\left(\sum_{p=0}^{\ell-1} M^{(p)}t^{\ell-1-p},\sum_{j=1}^d\gamma_j t^{\frac{j}{d}}\right)_{\mu_d}
=\sum_{p=0}^{\ell-1} \sum_{j=1}^d\left(M^{(p)},\gamma_j \right)_{\mu_d}t^{{\ell-1-p}+\frac{j}{d}},
$$
where $(-,-)_{\mu_d}$ denotes the multiplicity pairing, extended in the obvious way to $R(\mu_d)[t^{\frac{1}{d}}]$.
\end{prop}
\begin{proof}
Note that $M^{(p)}=0$ unless $0\leq p\leq\ell-1$ and that $\dim M^{(p)}[h,\exp(2\pi i\alpha]=0$ unless $\alpha=\frac{k}{d}$
for some $k\in\Z$. From these constraints, it is clear that $1\leq k\leq d\ell$, and that
 $Sp(\A)$ is a sum of terms of the form $f(t)(M^{(p)},\gamma_j)_{\mu_d}$,
where $f(t)$ is a (fractional) power of $t$. Moreover a simple calculation shows that
for $1\leq i\leq\ell$ and $1\leq j\leq d$, 
the coefficient of $(M^{(\ell-i)},\gamma_j)_{\mu_d}$ is $t^{\frac{(i-1)d+j}{d}}$.
\end{proof}


\begin{cor}\label{cor:spec} 
For $p \in \{0,1,...,\ell-1\}$ and $s \in \{0,1,...,d-1\}$,
we have, in the notation of \eqref{eq:theta3} and Proposition \ref{prop:spec-alt},
$c_{p,s}=\left(M^{(p)},\gamma_s\right)_{\mu_d}$

In particular, the restriction
$\theta (HD^{\Gamma_0}(F_0;u,1))$ contains exactly the same information as the spectrum $Sp(\A)$ and hence  it is determined by the combinatorics.
\end{cor}
\begin{proof}
From the remarks preceding Proposition \ref{prop:spec-alt}, \eqref{eq:theta3} may be written
$$
\theta(HD^{\Gamma_0}(F_0;u,1))=\sum_p \left(M^{(p)},\gamma_s\right)_{\mu_d}\gamma_su^p=\sum_{p,s}c_{p,s}\gamma_su^p,
$$
where $c_{p,s}=\left(M^{(p)},\gamma_s\right)_{\mu_d}$; this is the first statement. Hence Proposition \ref{prop:spec-alt}
may be written
$Sp(\A)=\sum_{p=0}^{\ell-1}\sum_{j=1}^d c'_{\ell-1-p,j}t^{p+\frac{j}{d}}$, where 
$c'(p,j)=\begin{cases}
c_{p,j}\text{ if }j\neq d\\
c_{p,0}\text{ if }j=d.\\
\end{cases}$
Thus $\theta(HD^{\Gamma_0}(F_0;u,1))$ contains precisely the same
information as $Sp(\A)$.
\end{proof}

Using the explicit formulas for the spectrum given in \cite{BS} , this proposition gives valuable information on the Poincar\'e-Deligne 
polynomial $PD^{\Gamma_0}(F_0;u,v))$.
We illustrate the use of the above results by computing the polynomial   $PD^{\Gamma_0}(F_0;u,v,t)$  in the
case where $G=\Sym_3$ or $G=\Sym_4$.

\begin{ex}\label{ex:PDsym3}

Consider first the case $G=\Sym_3$, when $\ell=2$. Example \ref{ex:typea2} implies that
$$P^{\Gamma_0}(F_0,t)=1\ot\gamma_0+(1\ot(\gamma_1 +\gamma_5) +
\rho\ot\gamma_0)t.$$
On the other hand, we know that the usual monodromy action $h^1$ on $H^1(F_0,\C)$ yields a decomposition
$$H^1(F_0,\C)=H^1(F_0,\C)_{ 1} \oplus H^1(F_0,\C)_{ \ne 1}.$$
Moreover, the first summand $H^1(F_0,\C)_{ 1}=\rho\ot\gamma_0$ is a pure Hodge structure of type $(1,1)$, and 
$H^1(F_0,\C)_{ \ne 1}=1\ot(\gamma_1 +\gamma_5) $ is a pure Hodge structure of weight $1$, see \cite{DL2}, Remark 1.4. (i) and 
Theorem 1.5 or \cite{DP}, Theorem 4.1. It follows that the only problem is to decide whether the subspace corresponding to
 $1\ot \gamma_1$ has Hodge type $(1,0)$ or $(0,1)$.

On the other hand, in this case it is easy to compute the spectrum $Sp(\A)$ since it coincides with the spectrum of the isolated 
hypersurface singularity $(x-y)(2x+y)(x+2y)=0$ and hence can be computed using the usual formulas, see for instance the 
formula (2.4.7) in \cite{DS}. It follows that
$$Sp(\A)=\left( \frac{t-t^{\frac{1}{3}}}{t^{\frac{1}{3}}-1}\right)^2=t^{\frac{2}{3}}+2t+t^{\frac{4}{3}}.$$
Applying Corollary \ref{cor:spec} with $\ell=2$, $s=1$, $p=0$ we get $\al=\frac{4}{3}$ and hence the coefficient of 
$\g_1u^0$ is $c_{0,1}=-m_{\frac{4}{3}}=-1$. This yields 
that $1\ot \gamma_1$ has Hodge type  $(0,1)$ and hence we get 
$$PD^{\Gamma_0}(F_0;u,v,t)=1\ot\gamma_0+(1\ot( \gamma_5u+\gamma_1v) +
\rho\ot\gamma_0uv)t.$$
\end{ex}

The results of \S\ref{sec:rank2} enable us to generalise this result to all rank 2 reflection groups.
\begin{ex}\label{ex:PDrk2}
 We use the notation of Proposition \ref{prop:mono}; thus $G$ is a rank two reflection group and $d=|\CA_G|$.
We shall show that 
\be\label{eq:pdrk2}
PD^{\mu_d}(F_0;u,v,t)=1\ot\gamma_0+\left(\sum_{i=2}^{d-1}(i-1)\gamma_i u+\sum_{j=1}^{d-2}(d-1-j)\gamma_j v +(d-1)\gamma_0 uv\right)t.
\ee
To see this, note that here we again have by \cite[(2.4.7)]{DS} that 
$$
Sp(\A)=\left( \frac{t-t^{\frac{1}{d}}}{t^{\frac{1}{d}}-1}\right)^2=
\sum_{i=1}^{d-2}it^{\frac{i+1}{d}}+(d-1)t+\sum_{j=1}^{d-2}jt^{1+\frac{d-1-j}{d}}.
$$
We now argue exactly as in the previous example to obtain \eqref{eq:pdrk2}. Note that it is evident that 
the coefficients of $u$ and $v$ are mutually contragredient in $R(\mu_d)$.
\end{ex}

The above argument may also be used to compute the weight polynomial of $F_0$ in the dihedral case.
The result is
\begin{prop}\label{prop:wtdih}
Let $G=G(p,p,2)$ be the dihedral group, as in \S\ref{sec:rank2}. 
If $p$ is odd, the weight polynomial of $F_0$ is given by
$$
W^{\Gamma_0}(F_0,t)=1\ot\gamma_0 + \left(\sum_{i=1}^{p-1}\left[\ve^{\ot(i+1)}+
\sum_{\overset{k=1}{k\neq i,p-i}}^{\frac{p-1}{2}} \chi_k \right] \ot \gamma_i \right)t  +
 \left(\sum_{k=1}^{\frac{p-1}{2}}\chi_k\ot\gamma_0  \right)t^2 . 
$$
\end{prop}
\begin{proof}

\end{proof}

Now we consider the more involved case $G=\Sym_4$, when $\ell=3$. 

\begin{prop}\label{prop:PDsym4} 
With the notation from Proposition \ref{prop:s4} one has the following.
$$
\begin{aligned}
PD^{\Gamma_0}(F_0;u,v,t)=&1\ot\gamma_0+[1\ot(\gamma_{10}u +\gamma_2v) +(\sigma+\rho)\ot\gamma_0uv]t
+\{(1\ot \gamma_9+\sigma\ot\gamma_{11})u^2\\
&+(1\ot \gamma_3+\sigma\ot\gamma_{1}) v^2)
+\rho\ot(\gamma_{10}u^2v+\gamma_2uv^2 +(\g_0+\g_6)u^2v^2)\}t^2.\\
\end{aligned}
$$
\end{prop}

\proof

The associated line arrangement in $\PP^2$ consists
of six lines, has only double and triple points and has exactly $\nu_3=4$ triple points. The formula given in Theorem 3 in \cite{BS} yields
$$Sp(\A)= t^{\frac{1}{2}} +3   t^{\frac{2}{3}}+2 t^{\frac{5}{6}}+6t+3t^{\frac{4}{3}}- t^{\frac{5}{3}}
-5t^2+2 t^{\frac{13}{6}}-  t^{\frac{7}{3}}+t^{\frac{15}{6}}.$$
Using Corollary \ref{cor:spec} as above we get
\begin{equation} 
\label{eq:theta4}
\theta (HD^{\Gamma_0}(F_0;u,1))=\gamma_0-\gamma_2-\gamma_4u-5\gamma_0u+\gamma_3u^2+\gamma_3+
6\gamma_0u^2+3\gamma_2u+3\gamma_4u^2+2\gamma_1+2\gamma_5u^2.
\end{equation}

As above, we know that the usual monodromy action on $H^1(F_0,\C)$ yields a decomposition
$$H^1(F_0,\C)=H^1(F_0,\C)_{ 1} \oplus H^1(F_0,\C)_{ \ne 1}.$$
Using the formula obtained in \eqref{eq:s4poin}, we see that  $H^1(F_0,\C)_{ 1}=(\sigma+\rho)\ot\gamma_0$ is a pure Hodge 
structure of type $(1,1)$, and $H^1(F_0,\C)_{ \ne 1}=1\ot(\gamma_2 +\gamma_{10}) $ is a pure Hodge structure of weight $1$. 
Moreover, one has a similar decomposition
$$H^2(F_0,\C)=H^2(F_0,\C)_{ 1} \oplus H^2(F_0,\C)_{ \ne 1}.$$
Here  $H^2(F_0,\C)_{ 1}=\rho\ot(\gamma_0+\gamma_6)$  is a pure Hodge structure of type $(2,2)$.
And $H^2(F_0,\C)_{ \ne 1}=1\ot(\gamma_3 +\gamma_9)+\rho\ot(\gamma_2+\gamma_{10})
+\sigma\ot(\gamma_1+\gamma_{11}) $ has Hodge weights $2$ and $3$ and the  characters $\g _i \in \hat \mu_{12}$ having 
weight $3$ must have order $3$ in $\hat \mu_6$,
 see \cite{DL2}.
 The above discussion combined with \eqref{eq:theta4} yield  the following two possibilities.
(The $\G$-representations occuring in these formulas should be considered as being $\G_0$-representations, exactly as in the
 discussion after the formula \eqref{eq:s4poin}.)
\be\label{eq:s4pd}
\begin{aligned}
PD^{\Gamma_0}(F_0;u,v,t)=&1\ot\gamma_0+[1\ot(\gamma_{10}u +\gamma_2v) +(\sigma+\rho)\ot\gamma_0uv]t
+\{1\ot(\gamma_9u^2+\gamma_3v^2)\\
+&\rho\ot(\gamma_{10}u^2v+\gamma_2uv^2 +(\g_0+\g_6)u^2v^2)
+\sigma\ot(\gamma_{11}u^2+\gamma_1v^2)\}t^2.\\
\end{aligned}
\ee
or
\be\label{eq:s4pdbis}
\begin{aligned}
PD^{\Gamma_0}(F_0;u,v,t)=&1\ot\gamma_0+[1\ot(\gamma_{10}u +\gamma_2v) +(\sigma+\rho)\ot\gamma_0uv]t
+\{1\ot(\gamma_9v^2+\gamma_3u^2)\\
+&\rho\ot(\gamma_{10}u^2v+\gamma_2uv^2 +(\g_0+\g_6)u^2v^2)
+\sigma\ot(\gamma_{11}u^2+\gamma_1v^2)\}t^2.\\
\end{aligned}
\ee
This indeterminancy is due to the fact that both characters $\g_3$ and $\g_9$ induce the same character $\g_3$ by restriction 
from $\mu_{12}$ to $\mu_{6}$.

To decide which of the two formulas above is the correct one, we construct a new  $\mu_{12}$-action on $F_0$ as follows. 
Consider the transposition $\tau=(1,2)$ and the group monomorphism $\mu_{12} \to \G_0$ given by $\zeta_{12} \mapsto (\tau,\zeta_{12})$. 
This morphism induces as above a ring morphism
\begin{equation} 
\label{eq:theta'}
\theta': R(\Gamma_0) \to R(\mu_{12}),
\end{equation}
by restriction of representations, namely 
$$\theta'(W\otimes \g_h)=(\dim W_0)\g_h+(\dim W_-)\g_{h+6}$$
where $W_{\pm}$ are the $\pm 1$-eigenspaces of $\tau$ in the representation $W$, and $h+6$ has to be computed modulo $12$. Using this we see that 
the coefficient of $u^2$ in the corresponding
Hodge-Deligne polynomial $\theta' (HP^{\Gamma_0}(F_0;u,v))$ is either
\be \label{eq:1a}
\g_5 +\g_9+\g_{11}
\ee
if formula \eqref{eq:s4pd} holds, or
\be \label{eq:1b}
\g_1 +\g_3+\g_{7},
\ee
if formula \eqref{eq:s4pdbis} holds. We study now this new action of $\mu_{12}$ on the cohomology group $H^2(F_0,\C)$ using a 
similar approach to section 5 in \cite{DL2}. First note that $\mu_{12}$-module $H^{2,0}(H^2(F_0, \C))$, which is  the coefficient 
of $u^2$ in  $\theta' (HP^{\Gamma_0}(F_0;u,v))$,
 is isomorphic to the $\mu_{12}$-module $H^{2,0}(H^2_c(F_0, \C))$ (compare with Corollary 1.2 in \cite{DL2}).

Next, let $V$ be the union of the six lines corresponding to the $A_3$-arrangement in $\PP^2$,
and choose $ Q_V(x_1,x_2,x_3)=0$ be a reduced equation for $V$ such that $\tau Q_V=-Q_V$.
One may take 
$$ Q_V(x_1,x_2,x_3)=(x_1-x_2)(x_1-x_3)(x_2-x_3)(2x_1+x_2+x_3)(x_1+2x_2+x_3)(x_1+x_2+2x_3).$$
Consider next the surface in $\PP^3$ given by
$$X_V: Q_V(x_1,x_2,x_3)-t^6=0.$$
The $\mu_{12}$-action on $F_0$ extends to a $\mu_{12}$-action on $X_V$ given by
$$\zeta_{12} \cdot (x_1:x_2:x_3:t)= (x_2:x_1:x_3:\zeta_{12}t).$$
According to formula \eqref{eq:cohoaction}, for a cohomology class $\al$, one has
$$\zeta_{12} \cdot \al= (h')^*(\al)$$
with $h'((x_1:x_2:x_3:t))= (x_2:x_1:x_3:\zeta_{12}^{-1}t).$

As $F_0=X_V \setminus V$, the long exact sequence of cohomology with compact supports yields
$$0 \to H^1(V) \to H_c^2(F_0) \to H^2(X_V) \to H^2(V) \to...$$
Since $H^1(V)$ has weights at most $1$, and $H^2(V)$ has type $(1,1)$ it follows an isomorphism of  $\mu_{12}$-modules 
$$H^{2,0}(H^2_c(F_0, \C))=H^{2,0}( H^2(X_V,\C)). $$
On the other hand, since $ H^2(X_V,\C)$ has weights at most $2$, it follows that  we have $H^{2,0}( H^2(X_V,\C))= Gr_F^2 H^2(X_V,\C))$, 
and hence we get an inclusion 
\be
\label{eqinclusion}
Gr_F^2 H^2(X_V,\C)) \to Gr_{F_{SS}}^2H^2(X_{\infty})
\ee
as in the exact sequence (5.3) in \cite{DL2}.
Here $X_{\infty}$ is a smooth surface in $\PP^{3}  $, nearby $X_V$, regarded as a generic fibre
in a 1-parameter smoothing $X_w$ of $X_V$. Moreover, $H^2(X_{\infty})$ is endowed with the Schmid-Steenbrink limit MHS, 
whose Hodge filtration will be denoted by $F_{SS}$.
The Hodge filtration  $F_{SS}$ on $H^2(X_{\infty})$, being the limit of the Deligne
Hodge filtration $F$ on $H^2(X_{w})$, yields an  isomorphism
\begin{equation} 
\label{eq20}
Gr_{F_{SS}}^2H^2(X_{\infty})=Gr_{F_{}}^2H^2(X_{w})
\end{equation}
of $\C$-vector spaces (i.e. equality of dimensions).
Note that our smoothing $X_w$ can be constructed in a $\mu_{12}$-equivariant way, e.g. just take $X_w$ to be the zeroset
 in $\PP^{3}$ of a polynomial of the form
$Q_V(x)-t^6+w(x_1^6-x_2^6+x_3^6+t^6)$. With such a choice, the isomorphism \eqref{eq20} becomes an equality in the 
representation ring $R(\mu_{12})$. Moreover, these representations can be explicitely determined, as explained in a similar context in Example 5.1 in \cite{DL2}.
A direct computation using rational differential forms yields
\begin{equation} 
\label{eqcompute}
Gr_{F_{}}^2H^2(X_{w})=\g_4+2\g_5+2\g_9+\g_{10}+4\g_{11}.
\end{equation}
Comparing this with the inclusion in \eqref{eqinclusion}, it follows that we are in the situation given by \eqref{eq:1a} and
 hence the formula \eqref{eq:s4pd} holds.
\endproof

Finally we consider the case of the braid arrangement $\A_4$, and forget (at least for the moment) the $Sym_5$-action on the Milnor fiber.

The $\A_4$ arrangement is associated with the natural $Sym_5$-action on $\C^5$, namely it consists of the $10$ hyperplanes 
$H_{ij}:x_i-x_j=0$ for $1 \leq i<j \leq 5$. To have an essential arrangement, we take the intersection with the hyperplane 
$H: x_1+...+x_5=0$ and get in this way our model for the arrangement $\A_4$ as a central essential arrangement in $\C^4$, 
consisting of the $10$ hyperplanes $H'_{ij}=H_{ij} \cap H$. To compute the spectrum we use the approach outlined in \cite{BS} 
and the explicit formulas given in \cite{Y}.

Note that the edges contained in the non normal crossing part of the central arrangement $\A_4$ in $\C^4$ are the following:

\begin{enumerate}

\item $10$ codimension $2$ edges $X$ where $3$ hyperplanes come together, e.g. as in the intersection 
$X= H'_{12}\cap H'_{23} \cap H'_{13}$; hence they have the multiplicity $m_X=3$.

\item $10$ codimension $3$ edges $Y$ which arise when an edge $X$ as above is cut by a transversal hyperplane, 
e.g. $Y=X\cap H'_{45}= H'_{12}\cap H'_{23} \cap H'_{13}\cap H'_{45}$; hence they have the multiplicity $m_Y=4$.

\item $5$ codimension $3$ edges $Z$ where $6$ hyperplanes come together, e.g. as in the intersection 
$X= H'_{12}\cap H'_{23} \cap H'_{13}\cap H'_{14}\cap H'_{24} \cap H'_{34}     $; hence they have the multiplicity $m_Z=6$.

\end{enumerate}

It is important to note that each edge $X$ of the first type $(i)$ contains exactly one edge $Y$ of the second type $(ii)$ and two
 edges $Z$ of the third type $(iii)$.
In the associated projective space $\PP^3$, the edges $X$ give rise to lines, while the edges $Y$ and $Z$ to points. Note that the 
edges of type $X$ and $Z$ are dense, while the edges of type $Y$ are not.

Using the above description and Theorem 1.1 in \cite{Y},  we get the following by a direct careful computation.

\begin{prop} \label{prop:specA4}
$$Sp(\A_4)=t^{\frac{4}{10}}+4t^{\frac{5}{10}}+5t^{\frac{6}{10}}+5t^{\frac{8}{10}}+6t^{\frac{9}{10}}+24t$$
$$+6t^{\frac{13}{10}}-t^{\frac{15}{10}}+t^{\frac{18}{10}}-26t^2$$
$$+
t^{\frac{22}{10}}+4t^{\frac{25}{10}}+6t^{\frac{27}{10}}+9t^3$$
$$+6t^{\frac{31}{10}}+5t^{\frac{32}{10}}+5t^{\frac{34}{10}}-t^{\frac{35}{10}}+t^{\frac{36}{10}}.$$
\end{prop}

Let $U_4$ be the complement of the corresponding hyperplane arrangement in $\PP^3$.
Then it is known that $b_0(U_4)=1$, $b_1(U_4)=9$, $b_2(U_4)=26$ and $b_3(U_4)=24$, and this explains the integral powers of 
$t$ in the above formula (recall that the spectrum uses the {\it reduced} cohomology ). In particular, $\chi(U_4)=-6$, so the alternating
sum of the multiplicities of any eigenvalue $\lambda \ne 1$ should be $-6$ by our Proposition \ref{prop:faithful}.

In fact, using the computations in Settepanella \cite{Se} or in our  Proposition \ref{prop:s5}, we already know these multiplicities. They are as follows.

\begin{enumerate}

\item the eigenvalue $-1$ occurs with multiplicity $2$ on $H^2(F_0)$ and with multiplicity $8$ on $H^3(F_0)$.

\item all the other eigenvalues $\exp(2 \pi i k/10)$ for $k=1,2,3,4,6,7,8,9$ occur only on the top cohomology group with multiplicity $6$.

\end{enumerate}
Using the spectrum computation above, we can describe the Hodge-Deligne $\mu_{10}$-equivariant polynomial of $F_0$ as follows.

\begin{cor}\label{cor:HodgeA4}

$$DP^{\mu_{10}}(F_0;u,v,t)=\g_0+9\g_0uvt+[\g_5u^2+\g_5v^2+26\g_0u^2v^2]t^2+$$
$$+[(\g_4+5\g_6+5\g_8+6\g_9 )u^3 +(6\g_3+\g_8 )u^2v+( \g_2+6\g_7)uv^2+(6\g_1+5\g_2+5\g_4+\g_6)v^3+$$
$$4\g_5u^3v+4\g_5uv^3+24\g_0u^3v^3]t^3.$$
\end{cor}

\proof

We will prove only the  claims concerning the character  $\g_5$, since the proof of the remaining claims is quite similar. Using the
 definition of the spectrum, the coefficient of $t^{\frac{15}{10}}$ gives
\begin{equation} 
\label{eq:specA}
\dim Gr_F^2H^3(F_0,\C)_{\gamma _5}-\dim Gr_F^2H^2(F_0,\C)_{\gamma _5}=-1.
\end{equation}
Since $H^2(F_0,\C)_{\gamma 5}$ is $2$-dimensional (and contains no elements of type $(2,2)$, see Theorem 1.3 in \cite{DL2}), the only 
possibility is $\dim Gr_F^2H^3(F_0,\C)_{\gamma _5}=0$ and on the other hand 
$\dim Gr_F^2H^2(F_0,\C)_{\gamma_ 5}=1$, which yields the  claim concerning $H^2(F_0)_{\g_5}$.
In a similar way, the coefficient of $t^{\frac{25}{10}}$ yields the equality $Gr_F^0H^3(F_0,\C)_{\gamma _5}=0$.

Next, $H^3(F_0,\C)_{\gamma _5}$ is $8$-dimensional and the coefficient of $t^{\frac{5}{10}}$
gives  the equality $\dim Gr_F^3H^3(F_0,\C)_{\gamma _5}=4$. This fact combined with the above vanishings shows that the eigenspace 
$H^3(F_0,\C)_{\gamma_ 5}$ has only Hodge classes of type $(3,1)$ and $(1,3)$. This proves the  claim concerning $H^3(F_0)_{\g_5}$.

\endproof

\begin{rk} \label{rk:NONcanc}

The spectrum of $\A_4$ has a non-cancellation property, since the only eigenvalue $\ne 1$ occuring in two distinct cohomology groups 
is $-1$, but the Hodge types $(p,q)$ have $p=0,2$ on $H^2$ and $p=1,3$ on $H^3$, so no cancellation is possible.

Note that this non cancellation property is quite subtle, since it involves either the knowledge of the multiplicities of $(-1)$, or the
 knowledge of the Hodge types of the eigenspaces $H^2(F_0,\C)_{\g_5}$ and $H^3(F_0,\C)_{\g_5}$ (and the fact that they give rise to distinct $p$'s).

The corresponding formula for the $\A_3$-arrangement is given below; the characters $\g_2$ and $\g_4$ occur in both $H^1$ 
and $H^2$. The corresponding Hodge types each have $p=1$
(once for $\g_4$ in $H^1$ and once for $\g_2$ in $H^2$), and similarly for $p=2$. This again yields a non-cancellation property for 
the corresponding spectrum.

It seems a difficult question to prove such a non cancellation result in general. To establish a relation to the non-cancellation property 
of the equivariant Euler characteristic $\chi^{\mu_m}(F/G)$ noticed in Remark \ref{rem:nocanc}, one should recall 
that $F/G=F_0/G_0$ as in Corollary \ref{cor:dps}, and hence the cohomology $H^*(F/G)=H^*(F_0/G_0)$ is a 
direct summand in the cohomology $H^*(F_0)$.

From this perspective, we see that the non-cancellation property of the spectrum is not a consequence of Remark \ref{rem:nocanc}
 (since it refers to larger eigenspaces) and does not imply Remark \ref{rem:nocanc} (since it contains
additional information coming from Hodge theory, which prevents possible cancellations at the level of $\chi^{\mu_m}(F)$.

\end{rk}

For the sake of completeness, we give below the corresponding $\mu_d$-equivariant Hodge-Deligne polynomials for the
 arrangements $\A_2$ and $\A_3$, to be derived in an obvious way from the results in Example \ref{ex:PDsym3} and Proposition \ref{prop:PDsym4}.

\begin{cor}\label{cor:HodgeA2+3} Let $\A_n$ be the essential braid arrangement in $\C^n$. Then one has the following.

\medskip

\noindent(i) The $\mu_3$-equivariant Hodge-Deligne polynomial for the braid arrangement $\A_2$ is given by
$$DP^{\mu_{3}}(F_0;u,v,t)=\g_0+[\g_2u+\g_1v+2\g_0uv]t.$$
\noindent(ii) The $\mu_6$-equivariant Hodge-Deligne polynomial for the braid arrangement $\A_3$ is given by
$$DP^{\mu_{6}}(F_0;u,v,t)=\g_0+[\g_4u+\g_2v+5\g_0uv]t+$$
$$+[(\g_3+2\g_5)u^2+(\g_3+2\g_1)v^2+3(\g_4u^2v+\g_2uv^2+2\g_0u^2v^2]t^2.$$

\end{cor}

\subsection{A purity result for eigenspaces in Milnor fiber cohomology} 

 Let $\A$ be an essential central arrangement of $d$ hyperplanes in $\C^{\ell}$, with  $\ell \geq 2$, given by a reduced 
equation $Q_0(x)=0$. Then clearly $H^m(F_0,\Q)_1$ and $H^{m}(F_0,\C)_{-1}$ are mixed Hodge substructures in 
$H^m(F_0,\Q)$ for any positive integer $m$. Moreover, for $\bt \in \mu_d$, $\bt \ne \pm 1$, the same is true for the subspace 
\begin{equation} 
\label{mhs1}
H^m(F,\C)_{\bt, \overline \bt}=H^m(F,\C)_{\bt} \oplus H^m(F,\C)_{ \overline \bt}=\ker[(h^m)^2-2Re(\bt)h^m+Id]
\end{equation}
which is in fact defined over $\R$ (as the last equality shows). For $\bt=-1$, we set
$H^m(F,\C)_{\bt, \overline \bt}=H^{m}(F,\C)_{-1}$ for uniformity of notation.

Let $D= Q_0^{-1}(0)=\cup_{H \in \A}H$ and for any dense edge $X$ of the arrangement $\A$ with $\dim X >0$ let $m_X$
 be the multiplicity of $X$, i.e. the number of hyperplanes in $\A$ containing $X$.

With this notation we have the following result, which complements our Proposition \ref{prop:mondeg} with Hodge theoretic information.

\begin{prop}
\label{prop1} 
Let $\bt \in \mu_d$, $\bt \ne 1$ be a $d$-root of unity such that $\bt ^{m_X} \ne 1$ for any dense edge $X$ with $\dim X >0$. 
Then the corresponding eigenspace $H^{\ell -1}(F_0,\C)_{\bt, \overline \bt}$ is a pure Hodge structure of weight $\ell-1$ and 
$H^m(F_0,\C)_{\bt, \overline \bt}=0$ for $m<\ell -1$.

In particular, if $\bt=\exp(-2\pi i\al)$ for some $\al \in \Q$, then the coefficients in the corresponding  spectrum $Sp(\A)$ have 
the following symmetry property: 
\begin{equation} 
\label{e2}
m_{\al}=m_{\ell-\al}.
\end{equation}

\end{prop}

\proof
For a point $x \in D$, $x \ne 0$, let $L_x=\cap_{H \in \A, x \in H}H$ and denote by $\A_x$ the central hyperplane arrangement induced by $\A$
on a linear subspace $T_x$, passing through $x$ and transversal to $L_x$. We may choose
$\dim T_x=\codim L_x$ and identify $x$ with the origin in the linear space $T_x$.
Let $h^*_x:H^*(F_x, \C) \to H^*(F_x, \C)$ be the corresponding monodromy operator at $x$. If $x \in X$, with $X$ a dense edge,
 it follows from the formula of the zeta-function $Z(t)$ of the monodromy given in Proposition \ref{prop:faithful} that the 
eigenvalues of $h^*_x$ are exactly the $(m_X)^{\text th}$-roots of unity. Indeed, it is known that an edge is dense if and only if the Euler
 characteristic of the projective complement associated to $\A_x$ is nonzero.

Then we apply Proposition 4.1 in \cite{DNag} and Proposition \ref{prop:mondeg}.
\endproof

\begin{ex}
\label{ex1} Consider the essential arrangement $\A_4$ in $\C^4$. If we consider the list of the dense edges given at the end 
of the subsection above, we see that $m_X=3$ and $m_Z=6$

It follows that for any $\bt \in \mu _{10}$ such that $\bt^6 \ne 1$, i.e. $\bt \ne \pm 1$, the corresponding eigenspace
 $H^3(F_0,\C)_{\bt, \overline \bt}$ is a pure Hodge structure of weight $3$. Moreover, we have the symmetry
 $m_{\al}=m_{4-\al}$ for all $\al \in \Q$ such that $2\al \not \in \Z$.

Note also that for $\bt =-1$, $H^2(F_0,\C)_{\bt, \overline \bt}$ is a pure Hodge structure of weight $2$. The following result generalises this property.

\end{ex} 

\begin{prop}
\label{prop2} 
Let $\bt \in \mu_d$, $\bt \ne 1$ be a $d$-root of unity such that $\bt ^{m_X} \ne 1$ for any dense edge $X$ with
 $\codim X < c(\bt)$ and there is at least one dense edge $Y(\bt)$ with $\codim Y(\bt) =c(\bt)$ such that $\bt ^{m_Y}=1$.
 Then the corresponding eigenspace $H^{d(\bt)}(F_0,\C)_{\bt, \overline \bt}$ is a pure Hodge structure of weight $d(\bt)$
 and $H^m(F_0,\C)_{\bt, \overline \bt}=0$ for $m<d(\bt)$,
with $d(\bt)=c(\bt)-1$.

\end{prop}

\proof

Let $E$ be a generic linear subspace of dimension $c(\bt)$ in $\C^{n+1}$. Then the dense edges of the hyperplane 
arrangement $\A |E)$ in $E$ obtained by taking all the traces $H \cap E$ for $H \in \A$ are exactly the intersections
 $X \cap E$ for $X$ a dense edge in $\A$ of codimension at most $c(\bt)$. The inclusion of Milnor fibers
$$\iota:  F_0(\A|E) \to F_0(\A)$$
induces an isomorphism $\iota^*: H^m(F_0(\A)) \to H^m(F_0(\A|E))$
for $m<\dim F_0(\A|E) -1=d(\bt)-1$ and a monomorphism for $m=d(\bt)$
preserving the mixed Hodge structures and compatible with the monodromy actions.

The result now follows from Proposition \ref{prop1}.
\endproof

As an application, we obtain a new proof of the following known result, see \cite{DL2}.

\begin{cor}\label{cor:linearr} Let $\A$ be central essential arrangement in $\C^3$. Then $H^1(F_0)_{\ne 1}$ is a pure Hodge structure of weight $1$.

\end{cor}

Indeed, $H^1(F_0)_{\ne 1}$ is a direct sum of pure Hodge structures of the type $H^1(F_0,\C)_{\bt, \overline \bt}$, 
each of weight $1$, since associated to dense edges of codimension $c(\bt)=2$.

\begin{ex}
\label{ex2} Consider the essential arrangement $\A_5$ in $\C^5$.
The dense edges $X$ of codimension $k$ are such that the corresponding arrangement $\A_X$ is the braid arrangement 
of type $\A_k$, for $k=1,2,3,4$, see Example \ref{ex:dense}. 

Take $\bt$ a primitive root of unity of order $5$.
Settepanella's computations in \cite{Se} imply that  $H^3(F_0,\C)_{\bt, \overline \bt}$ is $12$-dimensional, and our 
Proposition \ref{prop2} implies that it is a pure Hodge structure of weight $3$. 

However, it also follows from \cite{Se} that in $H^2(F_0)$ there is a $2$-dimensional eigenspace 
$H^2(F_0,\C)_{\bt, \overline \bt}$, where $\bt \ne 1$ is a cubic root of unity. Our Proposition \ref{prop2}
can say nothing about the corresponding mixed Hodge structure.
Indeed, such a $\bt$ is related to a codimension $2$ dense edge $Y(\bt)$.

\end{ex}

\end{document}